\documentclass[sn-mathphys-num]{sn-jnl}

\usepackage{graphicx}%
\usepackage{multirow}%
\usepackage{amsmath,amssymb,amsfonts,dsfont}%
\usepackage{amsthm}%
\usepackage{mathrsfs}%
\usepackage[title]{appendix}%
\usepackage{xcolor}%
\usepackage{textcomp}%
\usepackage{manyfoot}%
\usepackage{booktabs}%
\usepackage{algorithm}%
\usepackage{algorithmicx}%
\usepackage{algpseudocode}%
\usepackage{listings}%
\usepackage[nameinlink,capitalise]{cleveref}
\crefformat{equation}{#2(#1)#3}
\newcommand{\abs}[1]{\left| #1 \right|}

\newcommand{\one}{\mathds{1}}

\newcommand{\E}{\mathbb{E}}
\newcommand{\Eb}[1]{\mathbb E\left[#1\right]}
\newcommand{\R}{\mathbb{R}}

\newcommand{\rev}{\color{black}}

\theoremstyle{thmstyleone}%
\newtheorem{theorem}{Theorem}
\newtheorem{proposition}[theorem]{Proposition}%
\newtheorem{lemma}[theorem]{Lemma}

\theoremstyle{thmstyletwo}%
\newtheorem{corollary}{Corollary}

\newtheorem{remark}{Remark}%

\theoremstyle{thmstylethree}%
\newtheorem{example}{Example}%
\newtheorem{assumption}{Assumption}

\begin{document}

\title[Nonparametric estimation in McKean-Vlasov SDEs]{Polynomial rates via deconvolution for nonparametric estimation in McKean-Vlasov SDEs}


\author*[1]{\fnm{Chiara} \sur{Amorino}}\email{chiara.amorino@upf.edu}

\author[2]{\fnm{Denis} \sur{Belomestny}}\email{denis.belomestny@uni-due.de}

\author[3]{\fnm{Vytaut\.e} \sur{Pilipauskait\.e}}\email{vytaute.pilipauskaite@gmail.com}

\author[4]{\fnm{Mark} \sur{Podolskij}}\email{mark.podolskij@uni.lu}

\author[4]{\fnm{Shi-Yuan} \sur{Zhou}}\email{shi-yuan.zhou@uni.lu}

\affil[1]{\orgdiv{Barcelona School of Economics}, \orgname{Universitat Pompeu Fabra}, \orgaddress{\street{Ramon Trias Fargas 25-27 }, \city{Barcelona}, \postcode{ES-08005},  \country{Spain}}}

\affil*[2]{\orgdiv{Department of Mathematics}, \orgname{Dusiburg-Essen University}, \orgaddress{\street{Thea-Leymann-Str. 9}, \city{Essen}, \postcode{D - 45127}, \country{Germany}}}

\affil[3]{\orgdiv{Department of Mathematical Sciences}, \orgname{Aalborg University}, \orgaddress{\street{Thomas Manns Vej 23}, \city{Aalborg}, \postcode{DK-9220},  \country{Aalborg}}}

\affil[4]{\orgdiv{Department of Mathematics}, \orgname{University of Luxembourg}, \orgaddress{\street{6 avenue de la Fonte}, \city{Esch-sur-Alzette}, \postcode{L-4364},  \country{Luxembourg}}}


\abstract{This paper investigates the estimation of the interaction function for a class of McKean-Vlasov stochastic differential equations. The estimation is based on observations of the associated particle system at time $T$, 
considering the scenario where both the time horizon $T$ and the number of particles $N$
tend to infinity. Our proposed method recovers polynomial rates of convergence for the resulting estimator. This is achieved under the assumption of exponentially decaying tails for the interaction function. Additionally, we conduct a thorough analysis of the transform of the associated invariant density as a complex function, providing essential insights for our main results..}




\maketitle

\section{Introduction}\label{sec1}

The foundation of stochastic systems involving interacting particles and the development of nonlinear Markov processes, initially introduced by McKean in the 1960s \cite{McK66}, can be traced back to their roots in statistical physics, particularly within the domain of plasma physics. Over subsequent decades, the significance of these systems in probability theory has steadily grown. This area has witnessed the development of fundamental probabilistic tools, including propagation of chaos, geometric inequalities, and concentration inequalities. Pioneering contributions from researchers such as M\'el\'eard \cite{Mel96}, Malrieu \cite{Mal01}, Cattiaux et al. \cite{CatGuiMal08}, and Sznitman \cite{Sni91} have played a crucial role in shaping this field.

However, formulating a modern statistical inference program for these systems remained challenging until the early 2000s, with few exceptions, such as Kasonga's early paper \cite{Kas90}. Several factors contributed to this challenge. Firstly, the advanced probabilistic tools required for estimation were still under development. Secondly, the microscopic particle systems originating from statistical physics were not naturally observable, making the motivation for statistical inference less apparent. This situation began to change around the 2010s with the widespread adoption of these models in various fields where data became observable and collectable. Applications expanded into diverse fields, including the social sciences (e.g., opinion dynamics \cite{Cha17} and cooperative behaviors \cite{Can12}), mathematical biology (e.g., structured models in population dynamics \cite{Mog99} and neuroscience \cite{Bal12}), and finance (e.g., the study of systemic risk \cite{Fou13} and smile calibration \cite{Guy11}). Mean-field games have emerged as a new frontier for statistical developments, as evident in references \cite{9McK, 21McK, 33McK}. This transition has led to a growing need for a systematic statistical inference program, which constitutes the primary focus of this paper. Recently, this interest has manifested in two primary directions. On one front, statistical investigations are rooted in the direct observation of large interacting particle systems, as evidenced in works \cite{Us22, Che21, ComGen23, MaeHof, DelHof22, PavZan22a}. On the other front, statistical inference revolves around the observation of the mean-field limit, the McKean-Vlasov process, as exemplified in \cite{GenCat21a, GenCat21b, Imp21}. Concerning stationary McKean-Vlasov SDEs, the literature is relatively sparse. To the best of our knowledge, only a handful of references exist, including \cite{PavZan22b} and \cite{GenCat23}, which focus on the special McKean-Vlasov model without a potential term. In \cite{GenCat23b}, a more general model is explored.


This paper focuses on statistical inference for an interacting particle system described by the following stochastic differential equation:
\begin{equation}{\label{eq: model}}
X_t^{i, N} = X_0^i + B_t^i - \int_0^t V'(X_s^{i,N})\,ds - \frac{1}{2N} \sum_{j = 1}^N\int_0^t W'(X_s^{i, N} - X_s^{j,N})\, ds, \quad 1 \le i \le N,
\end{equation}
where the processes $B^i := (B_t^i)_{t \ge 0}$ are independent standard Brownian motions with unitary diffusion coefficient 
and $X_0^i$ are i.i.d. random variables with distribution $\mu_0(dx)$. The function $V$ is referred to as  the {\it confinement potential}, while $W$ (or $W'$) is the {\it interaction potential} (or interaction function respectively). Our primary objective is to estimate the interaction function $W'$ based on observations $X^{1,N}_T, ..., X^{N,N}_T$ of the particles, which are solutions of the system \eqref{eq: model}. Our approach hinges on the analysis of the associated inverse problem concerning the underlying stationary Fokker-Planck equation, relying on various results related to the probabilistic properties of the model.

Our research is closely connected to a recent study \cite{BPP} that discusses estimation of the interaction function with a specific semiparametric structure. The authors in \cite{BPP} develop an estimation procedure, demonstrating convergence rates that critically depend on the tail behavior of the nonparametric part of the interaction function $W$. Specifically, assuming a polynomial decay of the tails, they establish logarithmic convergence rates, proven to be optimal in that context. This naturally raises the question of whether polynomial rates can be achieved under a different set of conditions. Our paper aims to address this question, and our key finding is that, by assuming exponential decay of the interaction function $W$, we can introduce an estimator that achieves polynomial convergence rates, as demonstrated in Theorem \ref{th: estim beta}.

Compared to the framework proposed in \cite{BPP}, our model features some distinctions. The smoothness of the confinement potential $V$ emerges as a crucial factor influencing the achieved convergence rate. Additionally, the regularity of the invariant density $\pi$ of the associated McKean-Vlasov equation and the analysis of its Fourier transform $\mathcal{F}(\pi)$ are vital for establishing the asymptotic properties of our estimator. Notably, a lower bound on $\mathcal{F}(\pi)$ is required, presenting one of the primary challenges in our paper. We address this challenge using Hadamard factorization, leading to the desired lower bound under mild assumptions on the model. Furthermore, we provide an example demonstrating how a non-smooth confinement potential $V$
results in the Fourier transform $\mathcal{F}(\pi)$ exhibiting a polynomial decay. As another interesting situation, we consider the case of non-smooth potential $W$ and show that even in this case the Fourier transform of $\pi$ decays exponentially fast. Further technical tools include the extension of the Kantorovich-Rubinstein dual theorem to functions lacking Lipschitz continuity, presented in Theorem \ref{th: dual formulation}, and an extension of the uniform propagation of chaos in the $L^{2p}$-norm without convexity assumption for $W$, as demonstrated in Proposition \ref{prop: prop chaos L4}. \\

\noindent The structure of the paper is as follows. In Section \ref{s: model ass}, we introduce the model  assumptions. This section also offers a concise overview of interacting particle systems, and the relevant tools about Fourier and Laplace transforms. Crucially, we present key probabilistic results that lay the groundwork for our main findings.
Section \ref{s: main} is dedicated to formulating our primary statistical problem and the associated estimation procedure. Additionally, we establish upper bounds on the $L^2$ risk of the proposed drift estimator.
The proof of our main results is provided in Section \ref{S: proof main}, where we delve into the detailed verification of our key findings.
Finally, in Section \ref{s: transform}, we explore the sufficient conditions necessary to ensure that the transforms satisfy the requirements for our main results. All remaining proofs are collected in Section \ref{s: proof technical}.

\subsubsection*{Notation} 
All random variables and stochastic processes are defined on a filtered probability space $(\Omega, \mathcal{F}, (\mathcal{F}_t)_{t \ge 0}, \mathbb{P})$. Throughout the paper, we use the symbol $c$ to represent positive constants, although these constants may vary from one line to another. For any function $f: \R \rightarrow \R$, we denote its supremum as $\| f \|_{\infty} := \sup_{y \in \R} |f(y)|$. The notation $x_n \lesssim y_n$ signifies the existence of a constant $c > 0$, independent of $n$, such that $x_n \le c y_n$. The derivatives of a function $f$ are denoted as $f', f'', \ldots,$ or $f^{(k)}$, $k\geq 1$.
We use $C^k(\R)$ 
to denote the space of $k$ times continuously differentiable functions. 
For a complex number $z \in \mathbb{C}$, we denote its complex conjugate, real part, and imaginary part as $\overline{z}$, $\operatorname{Re}(z)$, and $\operatorname{Im}(z)$, respectively. For $a\in \R$, we use the notation 
$$\mathcal{L}_a:=\{y+\imath a: y \in \mathbb{R}\}.$$

\section{Model and assumptions}{\label{s: model ass}}
We start our analysis by introducing a set of assumptions on the confinement potential $V$ and the interaction potential $W$. It will become evident that the smoothness of $V$ plays a pivotal role in our asymptotic analysis. Our investigation encompasses two distinct scenarios: (a) The confinement potential of infinite smoothness takes the form $V(x) = \alpha x^2/2$ for a positive constant $\alpha$, (b) The confinement potential is expressed as $V(x) = \alpha x^2/2 + \widetilde{V}(x)$, where $\alpha > 0$ and $\widetilde{V}$ is a known function characterized by non-smooth features, as elucidated below.
We assume that $V$ and $W$ satisfy the following hypothesis: \\

\begin{assumption}\label{ass: beta}
The potentials $W\colon \R\to\R$ and $V\colon\R\to\R$ are {\rev non-negative functions} such that
\begin{itemize}
    \item The interaction potential $W\in C^2(\R)$ is even with bounded derivatives $W^\prime \in  
    L^1(\R)$ and $W''$ such that $\inf_{x\in \R} W''(x) = -C_W$ for some $C_W>0$. 
  
    \item The confinement potential $V$ is given by 
    $$
    V(x)=\frac{\alpha}{2}x^2 + \widetilde V(x), \quad \alpha + \inf_{x\in \R} \widetilde{V}''(x) = C_V>0
    $$
    where $C_V$ and $C_W$ satisfy the relation $C_V-C_W>0$. $\widetilde V$ is given by either

    \begin{enumerate}
    \item[A1.] $\widetilde V=0$, or
    \item[A2.] $\widetilde V$ is even, {\rev has polynomial growth} and there exists $J\in\mathbb N,\ J\ge 2$ such that $\widetilde{V}\in C^J(\R)$ and $\widetilde{V}\not\in C^{J+1}(\R)$. Furthermore, for each $2 \le j\le J$,  $\sup_{x\in \R}|\widetilde{V}^{(j)}(x)| = \tilde{c}_{j}<\infty$.
    \end{enumerate}
\end{itemize}
Additionally, the initial distribution admits a density $\mu_0 (dx) =\mu_0 (x) d x$ which satisfies
\begin{equation}\label{eq: cond exp moments}
    \int_{\R} \exp (c x ) \mu_0(x)\, dx <\infty, \quad \forall c\in\R, \qquad \int_{\R} \log(\mu_0(x)) \mu_0(x)\, d x < \infty.
\end{equation}

\end{assumption}

\paragraph{Discussion}
\begin{itemize}
\item[(i)] We would like to emphasize that we do not assume convexity of the interaction potential $W$, as is commonly done in most works. Instead, our assumption is that $W''$ is bounded below by a constant $-C_W<0$. While the geometric convergence of the distribution of the system \eqref{eq: model} to the invariant distribution for $t\to \infty$ now follows from \cite[Theorem~2.1]{CMV03}, Proposition \ref{prop: prop chaos L4} establishes a uniform (in time) propagation of chaos under the above assumptions.
Note also that A1 can be seen as a special case of A2, where $J=\infty$.\\
\item[(ii)] {\rev The presence of the quadratic term in the confinement potential $V$ is required to control the decay of the invariant density (cf. Lemma \ref{l: bound pi}). A similar semi-parametric assumption has been considered in  \cite{BPP}.} \\
\item[(iii)] {\rev In the scenario of a non-smooth confinement potential $\widetilde V$ as described in A2, we obtain an explicit polynomial lower bound on the decay of the characteristic function of the invariant density (cf. Example \ref{exmp}). This is vital for the statistical analysis.}
\end{itemize}

\subsection{Probabilistic Results}
The mean field equation associated to the interacting particle system introduced in \eqref{eq: model} is given by the $1$-dimensional McKean-Vlasov equation
\begin{equation}{\label{eq: mck}}
\overline{X}_t = \overline{X}_0 + B_t - \int_0^t V'(\overline{X}_s)\, ds - \frac{1}{2}\int_0^t (W' \star \mu_s) (\overline{X}_s)\, ds, \qquad t \ge 0,
\end{equation}
where $\mu_t(dx) := \mathbb{P}(\overline{X}_t \in dx)$ and 
$$(W' \star \mu_t)(x) = \int_{\R} W'(x - y) \mu_t(dy), \qquad x\in \R, \qquad t \ge 0.$$
Under \cref{ass: beta}, existence and uniqueness of strong solutions of \cref{eq: model} and \cref{eq: mck} {follow as in \cite{existence, CMV03, Mal03}}.
Additionally, the measure $\mu_t$ possesses a smooth Lebesgue density and the McKean-Vlasov equation admits an invariant density $\pi$ solving the stationary Fokker-Planck equation
\begin{equation}
    \frac 12 \pi^{\prime\prime} = -\frac{\mathrm d}{\mathrm dx}\Bigl(\Bigl(V^\prime + \frac12 W^\prime \star \pi\Bigr)\pi\Bigr),
\end{equation}
which means that $\pi$ is given by
\begin{equation}{\label{eq: pi}}
\pi (x) = \frac{1}{Z_{\pi}} \exp \left( - 2V(x) - W \star \pi (x) \right), \quad x \in \mathbb{R}, 
\end{equation}
with a normalizing constant
$$
Z_{\pi} := \int_{\R} \exp \left(- 2V(x) - W \star  \pi (x)\right) d x < \infty.
$$
The invariant density can be upper and lower bounded under our assumptions, according to the following lemma. Its proof can be found in Section \ref{s: proof technical}.

\begin{lemma}{\label{l: bound pi}}
   Suppose that \cref{ass: beta} holds true. Then, for any $x \in \R$, there exist two constants $c_1, \, c_2 > 0$ such that
   $$c_1 \exp(- \widetilde{C} x^2) \le \pi(x) \le c_2 \exp(- C_V x^2),$$
   with $\widetilde{C} := \alpha + \tilde{c}_2$ and $C_{{V}}$ and $\tilde{c}_2$ are as in \cref{ass: beta}.
    Furthermore, it holds that
    \begin{equation*}
        |\pi^{(n)}(x)| \le c (1 + |x|)^n \exp(- C_V x^2), \quad ~0\le n\le J.
    \end{equation*}

\end{lemma}

\noindent A crucial tool in our estimation procedure consists {in the result which combines uniform in time propagation of chaos
for \eqref{eq: model} with convergence to the equilibrium of \cref{eq: mck}}. 
In order to state it we start by introducing the Wasserstein metric. The \textit{Wasserstein $p$-distance} between two measures $\mu$, $\nu$ on $\R$ is defined by 
$$W_p(\mu, \nu) := \left( \inf_{X \sim \mu, Y \sim \nu} \E[|X - Y|^p]\right)^\frac{1}{p},$$
where the infimum is taken over all the possible couplings $(X, Y)$ of random variables $X$
and $Y$ with respective laws $\mu$ and $\nu$. 

{The following result combines Point (iv) of Theorem 2.1 in \cite{CMV03} with Theorem 5.1 in \cite{Mal03} or Theorem 3.1 in \cite{CatGuiMal08}, adapted to the current framework, see also our \cref{prop: prop chaos L4} below.}

\begin{theorem}{\label{th: prop chaos}}
Let $\overline{X}^i$, $1 \le i \le N$, be i.i.d.\ copies of the process $\overline{X}$ defined in \eqref{eq: mck} so that every $\overline{X}^i$ is driven by the same Brownian motion as the i-th particle of the system \eqref{eq: model} and they are equal at time $0$. Denote by 
$$\Pi_{N,T} = N^{-1} \sum_{i = 1}^N \delta_{X_T^{i,N}}$$
the empirical distribution of the particle system $X_T^{i,N}$ for $1 \le i \le N$, and by $\Pi$ the law associated to the invariant density $\pi$. Under \cref{ass: beta} there exist a constant $c > 0$ independent of $N$ and $T$ such that 
$$\sup_{t \ge 0} \E \left[|X_t^{i,N} - \overline{X}_t^{i}|^2 \right] \le c N^{-1}$$
and 
$$\E [ W_1^2(\Pi_{N,T}, \Pi) ] \le c\left( N^{-1} +  \exp(- \lambda T)\right) =: c N_T^{-1},$$
where $\lambda = C_V - C_W >0$,  and $C_V$ and $C_W$ have been introduced in \cref{ass: beta}.
\end{theorem}

\noindent The aforementioned bound asserts that the invariant distribution $\Pi$ of the mean field equation can be accurately approximated by the empirical measure $\Pi_{N,T}$, while providing an associated error bound for this approximation. In the upcoming estimation procedure all convergence rates will be measured in terms
of $N_T$. 
In the following, we will present and prove a similar result, but from a dual perspective, involving the Laplace transforms of $\Pi_{N,T}$ and $\Pi$.
We opt to replace the use of Fourier transform, as seen in \cite{BPP}, with the Laplace transform.  This choice is made with a similar intent as the authors of \cite{BelGol}, allowing for greater flexibility and the inclusion of diverse scenarios. \\
\\
Before we proceed further, let us now introduce some notation and properties concerning the Laplace transform. 
For any locally integrable function $\phi$, we define the bilateral Laplace transform as follows:
\begin{equation}{\label{eq: def laplace}}
\widehat{\phi}(z) 
:= \int_{\R} \phi(t) \exp(-z t) dt. 
\end{equation}
{\rev Moreover, we define the Fourier transform on $\mathbb{C}$ as follows: 
$$\mathcal{F}(\phi(z)) = \int_{\R} \phi(z) \exp(i zt) dt. $$}
The Laplace transform $\widehat{\phi}(z)$ is an analytic function within the convergence region $\Sigma_\phi$, which typically takes the form of a vertical strip in the complex plane:
$$\Sigma_\phi := \left \{ z \in \mathbb{C} \, : \, x_\phi^- \leq \operatorname{Re}(z) \leq x_\phi^+ \right \}$$
for some $x_\phi^-$, $x_\phi^+$ such that $- \infty < x_\phi^- \leq x_\phi^+ < \infty$. The convergence region $\Sigma_\phi$ can degenerate to a vertical line in the complex plane, in such case it is $\Sigma_\phi := \left \{ z \in \mathbb{C} \, : \, \operatorname{Re}(z) = x_\phi \right \}$, with $x_\phi \in \R$. 
If $\phi$ is a probability density, then the imaginary axis always belongs to the convergence region $\Sigma_\phi$. In this case the Fourier transform of $\phi$
$$\widehat{\phi}(-\imath y) = \mathcal{F}(\phi)(y) := \int_{\R} \phi(t) \exp(\imath y t) dt, \qquad y \in \R,$$
is the characteristic function of $\phi$. This degenerate case pertains to distributions whose characteristic function lacks the ability to be analytically extended to a strip surrounding the imaginary axis in the complex plane.

We are now ready to state a propagation of chaos type result for the Fourier transforms that is reminiscent of Theorem \ref{th: prop chaos} by means of the Kantorovich-Rubinstein theorem. In particular, we can demonstrate that the transform of $\Pi$ can also be effectively approximated by the transform of the empirical measure $\Pi_{N,T}$. The proof can be found in Section \ref{S: proof main}.

\begin{theorem}{\label{th: dual formulation}}
Under \cref{ass: beta}, there exists a constant $c > 0$ such that, for any $z \in \mathcal{L}_{\pm a}$, 
$$\E [ |\mathcal{F}(\Pi)(z)-\mathcal{F}(\Pi_{N, T})(z) |^2 ] \le c |z|^2 N_T^{-1}.$$    
\end{theorem}

\begin{remark} \rm
Note that for $z \in \mathbb{R}$, this theorem is a direct implication of the Kantorovich-Rubinstein dual formulation, applicable to $1$-Lipschitz functions. However, when dealing with $z \in \mathcal{L}_{\pm a}$, the function $\exp(iz)$ no longer possesses Lipschitz continuity, thereby rendering the use of the Kantorovich-Rubinstein dual formulation unfeasible. This motivates us to establish this analogous formulation. \qed
\end{remark}

\noindent Some crucial results will be instrumental in the derivation of the theorem above. Specifically, we present an extension of the propagation of chaos theorem, as stated in Proposition \ref{prop: prop chaos L4} below, in the $L^{2p}$ norm, for $p \ge 1$. 
For the detailed proof of this result, we refer to Section \ref{s: proof technical}. \qed

\begin{proposition}{\label{prop: prop chaos L4}}
Under \cref{ass: beta} for any $p\ge 1$ there exists a constant $c > 0$ such that uniformly in $N$ and $i$,
$$\sup_{t \ge 0} \E\left[|X_t^{i,N} - \overline{X}_t^{i}|^{2p}\right] \le c N^{-p}.$$
\end{proposition}

\section{Statistical Framework and Main Results}{\label{s: main}}

\subsection{The Estimation Procedure}
Suppose we observe the data $X_T^{1,N}, \ldots , X_T^{N,N}$ in the asymptotic framework where both $N$ and $T$ tend towards infinity. Our goal is to estimate the interaction function $W'$. In particular, we will propose an estimator $W'_{N,T}$ and study its performance by considering the associated mean integrated squared error, aiming to achieve polynomial convergence rates. 

The estimation procedure is semiparametric in the sense that it consists of four different steps, involving both parameter estimation and nonparametric estimation techniques.
\begin{enumerate}
    \item The first step consists in the estimation of the derivative of the log-density which we denote as $l(y)$: 
    $$l(y) := (\log \pi)'(y) = \frac{\pi'(y)}{\pi(y)}, \qquad y \in \R.$$
    This will be achieved by the introduction of kernel estimators for both $\pi$ and $\pi'$. Let $K$ be a smooth kernel of order $m\ge 2$, that is 
$$\int_{\R} K(x)\, dx = 1, \qquad \int_{\R} x^j K(x)\, dx = 0, \quad j=0,\ldots, m-1, \qquad \int_{\R} x^m K(x)\, dx \neq 0.$$
It is worth noting the choice of the kernel order, denoted by $m$, is flexible and can be determined by the statistician. As we will see later on, the choice of $m$ is determined by the regularity of $V$: a smooth confinement potential allows us to choose an arbitrary $m\in\mathbb N$. On the other hand, if $V$ is non-smooth
as described in \cref{ass: beta}, A2, we need to additionally assume that $m \le J$. Indeed, this is a standard restriction on the order of the kernel when estimating non-smooth functions.
Let us also introduce the bandwidths $h_i := h_{i,N,T}$, $i=0,1$, which satisfy $h_i\to 0$ for $N, T$ going to $\infty$ and the notation $K_h(x)=\frac{1}{h} K(\frac{x}{h})$. 
Then, we can define the kernel estimators $\pi_{N,T}$ and $\pi_{N,T}'$ for $\pi$ and $\pi'$, respectively, as below:
$$\pi_{N,T}(y) := \frac{1}{N} \sum_{i =1}^N K_{h_0} (y - X_T^{i,N}), \qquad \pi'_{N,T}(y) := \frac{1}{N h_1} \sum_{i =1}^N K'_{h_1} (y - X_T^{i,N}), \quad y \in \R. $$
An estimator for the derivative of the log-density $l(y)$ is then given by 
$$l_{N,T}(y) := \frac{\pi'_{N,T}(y)}{\pi_{N,T}(y)} 1_{\{\pi_{N,T}(y) > \delta\}},$$
where $\delta = \delta_{N,T}\to 0$ as $N, T \rightarrow \infty$. 
\item In the second step, we estimate the parameter $\alpha > 0$ appearing in the confinement potential. This is based on the identity 
\begin{equation*}
    l(y) = - 2\alpha y - 2\widetilde{V}'(y) - W' \star \pi(y)
\end{equation*}
and on a contrast function method. Indeed, since $W' \in L^1(\R)$, we know that $\abs{W' \star \pi(y)} \rightarrow 0$ as $|y| \rightarrow \infty$, which allows us to construct a minimal contrast estimator for $\alpha$. In particular, for any $\epsilon \in (0,1)$ arbitrarily small, we can introduce an integrable weight function $w$ with support on $[\epsilon, 1]$ and a parameter $U = U_{N,T}$, which satisfies $U\to \infty$ for $N, T \rightarrow \infty$. Then, we can define the estimator $\alpha_{N,T}$ for $\alpha$ as
\begin{equation*}
    \alpha_{N,T} :=  \underset{\alpha \in \R}{\operatorname{argmin}} \int_{\R} \left(l_{N,T}(y) + 2 \alpha y +  2\widetilde{V}'(y)\right)^2 w_U(y)\, dy,
\end{equation*}
{\rev where $w_U(\cdot) := (1/U)w(\cdot/U)$.}
\item Using the results in the previous step we can construct an estimator $\Psi_{N,T}$ of $\Psi:= - W' \star \pi$. Indeed, given the estimators $l_{N,T}$ and $\alpha_{N,T}$ constructed above, we can define
$$\Psi_{N,T}(y) := \left(l_{N,T}(y) + 2\alpha_{N,T} \, y + 2\widetilde{V}'(y)\right)1_{\{|y| \le \epsilon U\}}, \qquad y \in \R.$$
\item The last step consists in applying the deconvolution and the inverse Laplace transform to obtain an estimator for $W^\prime$. Note that because we want to consider values of $\mathcal{F}(\Pi)$ in the strip of analyticity (e.g. to escape zeros of $\mathcal{F}(\Pi)$), we can not use the standard regularization techniques of deconvolution problems consisting  of cutting off the estimates outside bounded intervals on real line, see \cite{JJ09}. Instead, we are going to employ Tikhonov-type regularization.    More specifically,  for some $a\ge0$, choose a sequence of entire functions $\rho_{N, T}$ such that for $z \in \mathcal{L}_{\pm a}$, $\left|\mathcal{F}\left(\Pi_{N, T}\right)(z)+\rho_{N, T}(z)\right| \geq \varepsilon_{N, T}>0$, where $\varepsilon_{N, T} \rightarrow 0$ for $N,T \rightarrow \infty$. Then, we define the estimator $W_{N, T}^{\prime}$ via
\begin{equation}
\label{eq:tikhonov-reg}
\mathcal{F}\left(W_{N, T}^{\prime}\right)(z):=-\frac{\mathcal{F}\left(\Psi_{N, T}\right)(z)}{\mathcal{F}\left(\Pi_{N, T}\right)(z)+\rho_{N, T}(z)} .
\end{equation}
Observe that the right hand side is well-defined since $\Psi_{N,T}\in L^1(\R) \cap L^2(\R)$. Additionally, for all $z\in \mathcal{L}_{\pm a}$, $|\mathcal{F} (W_{N, T}^{\prime} )(z)|\le \varepsilon_{N,T}^{-1} \abs{\mathcal F(\Psi_{N,T})(z)}$, such that $W_{N, T}^{\prime}$, as the inverse Fourier transform of the right hand side, is well-defined.
\end{enumerate}

\begin{remark} \rm
The use of Tikhonov-type regularization in \eqref{eq:tikhonov-reg}, instead of more common hard-thresholding-type regularizations in deconvolution problems, is motivated by the fact that we want to consider the estimate \(\mathcal{F}\left(W_{N, T}^{\prime}\right)\) on the complex plane. This, in turn, makes it possible to weaken our assumptions on the zeros of \(\mathcal{F}\left(\Pi\right)\) and, in particular, allows for zeros on the real line.  For example, the Fourier transform of the density $\pi(x)\propto\exp(-\sum_{k=1}^p c_k x^{2k})$ with $c_k> 0$ corresponding to polynomial potential $W$, has only real zeros for all natural $p>1$ (see Theorem~20 in \cite{deBruijn}) and hence it is reasonable to consider integration contours $\mathcal{L}_a$ with $a>0$ in this situation. We refer to \cite{BelGol} for a similar approach, where the problem of singularities in the deconvolution estimates was overcome by considering general integration contours in the complex plane when computing the inverse Fourier transform. In fact, Tikhonov regularization, also known as ridge regularization, is a commonly used technique to stabilize the solution of inverse problems, including deconvolution problems, see \cite{ridgedeconv}. 
 \qed
\end{remark}

\noindent It is natural to draw a comparison between our proposed estimation approach with the one presented in \cite{BPP}, particularly in scenarios where the interacting drift exhibits polynomial tails. Although the overall steps in the estimation process share similarities, our context introduces several novel considerations.

Our estimation procedure can be divided into two parts, depending on whether we are addressing the first case with infinite smoothness, where $\widetilde{V} = 0$, or the less smooth case where $\widetilde{V}$ adheres to condition A2. The model distinction arises from the absence of the potential function $V(x)$ in \cite{BPP}. Instead, the interaction potential in \cite{BPP} comprises two components: the potential, encompassing trigonometric and polynomial functions, and the non-parametric component of $W$. Despite this difference, it does not significantly impact the estimation procedure. The parametric component in \cite{BPP} plays a role similar to the confinement potential in our context, and both are estimated through a contrast function, yielding comparable results in Steps 1 and 2 for both cases.

The deviation becomes evident in Step 3, where the constraint on exponential tails of $W'$ results in polynomial convergence rates (see Theorem \ref{th: estim psi}), a contrast to the logarithmic convergence rates for $\Psi$ in \cite{BPP} due to polynomial tails of $W'$.

The divergence continues into the fourth and final step, where the estimation procedure takes on entirely different forms. The primary challenge lies in analyzing the joint decay of the transforms of $W'$ and $\pi$, introducing the condition 
$$
\int_{\mathcal{L}_a} \left| \frac{\mathcal{F} (W^{\prime} )(z)}{\mathcal{F}(\Pi)(z)} \right|^2 \,dz<\infty.
$$ Analyzing such a condition proves to be a complex task. Notably, the analysis hinges on studying the zeros of the transform of $\pi$, leading us to use the Laplace transform instead of the Fourier transform. It is worth noting that selecting $a=0$ in $\mathcal{L}_a$ allows obtaining the Fourier transform from the Laplace transform. A comprehensive explanation regarding the fulfillment of the mentioned constraint can be found in Section \ref{s: transform}.

\subsection{Main Results: Convergence Rates}
{\noindent Let us start with the first step, which consists in the estimation of $l$. We remark again that the order of the kernel can be chosen arbitrarily in the case of $\widetilde V=0$, whereas for $\widetilde V$ as in \cref{ass: beta}, A2, we require the condition $m \le J$. In the sequel the bandwidth $h_0$ for $\pi_{N, T}$ is chosen as 
\begin{equation}{\label{eq: h0}}
h_0 := N_T^{-\frac{1}{2(m + 1)}}.
\end{equation}
Similarly, for the estimation of $\pi_{N, T}$ we choose
\begin{equation}{\label{eq: h1}}
h_1: = N_T^{-\frac{1}{2(m + 2)}}.
\end{equation}
Finally, the threshold parameter is chosen as 
\begin{equation}{\label{eq: delta}}
\delta := \frac{c_1}{2} \exp(- \widetilde{C} U^2),
\end{equation}
where $c_1$, $\widetilde{C}$ are the constants appearing in Lemma \ref{l: bound pi}. 
We remark that the choice of $h_0$ and $h_1$ is the same as in \cite{BPP} while the choice of $\delta$ is due to our lower bound on $\pi$ as presented in Lemma \ref{l: bound pi}. 

\begin{proposition}{\label{prop: bound l}}
Let $h_0$, $h_1$ and $\delta$ be as above and let $U \ge 1$. Assume that \cref{ass: beta} holds. Then it is 
$$\sup_{|x| \le U} \E\left[|l_{N, T}(x) - l(x)|^2\right] \lesssim \exp (2 \widetilde{C} U^2) \left( N_T^{- \frac{2m}{2(m + 2)}} + U N_T^{- \frac{2m}{2(m + 1)}}\right).$$
\end{proposition}

\noindent We now proceed to estimate $\alpha$ in Step 2, employing the estimator $\alpha_{N, T}$. In our current context, this step is less challenging than it was in the previous work cited as \cite{BPP}, thanks to the specific model under consideration.
In particular, the estimation of $\alpha$ essentially involves simplifying Step 2 from \cite{BPP} to the scalar case, where an additional potential $\widetilde{V}$ has been introduced. However, this potential is already known, and it can be chosen in a way such that it does not contribute to the convergence rate at this stage of the analysis.
}

\begin{theorem}{\label{th: estim psi}}
Let $U \ge 1$ and recall that $m$ is the order of the kernel $K$. If \cref{ass: beta} holds then, for any $\epsilon \in (0,1)$,
{\begin{align}{\label{eq: bound psi}}
{\rev \left( \E  \left[ \int_\R |\Psi_{N, T} (y) - \Psi (y)|^2 dy \right] \right)^\frac{1}{2}} & \lesssim \exp(\widetilde{C} U^2) U^{\frac{1}{2}} \left(N_T^{- \frac{m}{2(m + 2)}} + U N_T^{- \frac{m}{2(m + 1)}}\right)  \\
& + \frac{ 2\exp(- C_V (\epsilon \frac{U}{2})^2)}{\epsilon U} +  \left( \int_{|y| > \frac{\epsilon U}{2}} |W'(y)|^2 dy \right)^\frac{1}{2}.  \nonumber 
\end{align} }
\end{theorem}
\noindent The dependence of the convergence rate for $\Psi_{N,T}$ on the tail behavior of the function $W'$ is evident in Theorem \ref{th: estim psi}. In particular, when the tails of $W'$ exhibit exponential decay, the subsequent corollary, which directly follows from the previous result, provides a precise bound.

\begin{corollary}{\label{cor: rate psi}}
In the setting of previous theorem, let $p > {C_V}$ and assume that 
\begin{equation}
\label{eq: cond tails beta}
\limsup_{x\to\infty} \exp(2 px^2)\int_{|y|>x} |W'(y)|^2 \, dy < \infty.
\end{equation}
Then, for any $\epsilon \in (0,1)$, choosing $U^2=c_u\log(N_T)$, where
\begin{equation}{\label{eq: def cu}}
c_u = \frac{m}{2(m + 2)} \frac{1}{(\widetilde{C} + C_V \frac{\epsilon^2}{4})}
\end{equation}
gives
\begin{equation}{\label{eq: pol psi}}
{\rev \left( \E  \left[ \int_\R |\Psi_{N, T} (y) - \Psi (y)|^2 dy \right] \right)^\frac{1}{2}} \lesssim (\log N_T)^{\frac14}N_T^{-\gamma},
\end{equation}
where
\begin{equation}{\label{eq: gamma}}
    \gamma=\frac{m}{2(m+2)} \frac{1}{1 + \frac{4 \widetilde{C}}{\epsilon^2 C_V}}.
\end{equation}

\end{corollary}

\begin{remark} \rm
One might question why the convergence rate above depends on the auxiliary parameter $\epsilon \in (0,1)$ introduced in the second step of our estimation procedure. By following the proof of Theorem \ref{th: estim psi}, it is easy to verify that when the value of $\alpha$ is known and does not require pre-estimation, the results still hold with $\epsilon = 1$. However, when estimating the parameter $\alpha$, we lose the option of setting $\epsilon = 1$ and can only use $\epsilon \in (0,1)$, which slightly affects our convergence rate. The optimal choice is to take $\epsilon$ as close to $1$ as possible, resulting in a convergence rate $\gamma$ equal to $$\frac{m}{2(m+2)} \frac{1}{1 + \frac{4 \widetilde{C}}{C_V }} - \tilde{\epsilon}$$ for any arbitrarily small $\tilde{\epsilon}$. \qed
\end{remark}
 The estimation of $\Psi$ leads us to the estimation of $W'$ as stated in the following theorem. In this context, an assumption regarding the decay of the transforms of $\pi$ and $W'$ will be crucial.
\begin{assumption}
    \label{ass: FTs} Recall that $\mathcal{L}_a = \left \{ y + \imath a : y \in \R \right \}$ for some $a \ge 0$.
    There exists an $\overline{a} \ge 0$ such that
    \begin{equation} \label{eq: ass transforms}
        \int_{\mathcal{L}_{\pm \overline{a}}} \left|\frac{\mathcal{F}\left(W^{\prime}\right)(z)}{\mathcal{F}(\Pi)(z)}\right|^2 d z<\infty.
    \end{equation}
\end{assumption}

\noindent We will further analyze this assumption in \cref{s: transform}. In such section, we will employ tools from complex analysis to establish a lower bound on $\mathcal{F}(\Pi)(z)$, allowing us to verify that \cref{ass: FTs} holds true in specific situations. Additionally, at the conclusion of Section \ref{s: transform}, we provide an example illustrating how a non-smooth confinement potential implies polynomial decay in the transform of the invariant density. This example assists us in verifying the validity of \cref{ass: FTs} as mentioned earlier. \\
\\
In the sequel, $\varepsilon_{N, T}$ is chosen as 
$$\varepsilon_{N, T}= \exp\left(\frac{a \epsilon}{2}(c_u \log N_T)^\frac{1}{2}\right) (\log N_T)^\frac{1}{4} N_T^{- \frac{\gamma}{2}}.$$
From a quick look at Theorem \ref{th: estim beta}, it is easy to see that it provides the final convergence rate for the estimation of $W'$.

\begin{theorem}{\label{th: estim beta}}
 Assume that \cref{ass: beta} and \cref{ass: FTs} hold true for some ${\rev \bar{a}} \ge 0$. Then
\begin{align*}
 {\rev \left( \E\left[ \int_{\R} |W^{\prime}_{N,T}(y) - W^{\prime}(y) |^2 dy\right] \right)^\frac{1}{2}} \le c \exp\left(\frac{{\rev \bar{a}} \epsilon}{2}(c_u \log N_T)^\frac{1}{2}\right) (\log N_T)^\frac{1}{4} N_T^{- \frac{\gamma}{2}},
\end{align*}
where $c_u$ is as in \eqref{eq: def cu} and $\gamma$ as in \eqref{eq: gamma}.
\end{theorem}

\begin{remark} \rm
Observe that when \cref{ass: FTs} holds for ${\rev \bar{a}}=0$ 
the convergence rate found in Theorem \ref{th: estim beta} is $(\log N_T)^\frac{1}{4} N_T^{- \frac{\gamma}{2}}$. Thanks to Hadamard representation we will see in Section \ref{s: transform} we can obtain such result if $\mathcal{F}(\Pi)(z)$ does not have zeros on the real line (see Theorem \ref{th: lower bound} below).  
In Theorem \ref{th: estim beta}, we consider a more general case, where we allow the function $\mathcal{F}(\Pi)(z)$ to have real zeros. The crucial condition is indeed that there exists at least one line parallel to the real axis on which such function does not have any zeros. Let us stress that this condition is much weaker then one requiring no zeros of $\mathcal{F}(\Pi)(z)$ on the real axis (${\rev \bar{a}}=0$) since  there is a large class of densities with Fourier transforms vanishing only on the real line, see, e.g. \cite{bruijn}. We refer to Section~\ref{s: transform} for more details. \qed
\end{remark}

\begin{remark} \rm
Even when \cref{ass: FTs} holds for ${\rev \bar{a}} \neq 0$, the observed convergence rate, as determined in the aforementioned theorem, remains polynomial. This outcome arises from the dominance of the polynomial term over the exponential term in Theorem \ref{th: estim beta}. \qed

\end{remark}

\begin{remark} \rm
It is important to highlight that the convergence rate outlined in Theorem \ref{th: estim beta} is contingent upon the smoothness of the confinement potential $\widetilde{V}$. Specifically, in the scenario of smooth potentials, we have the flexibility to set the parameter $m$ to be arbitrarily large. Additionally, it's noteworthy to mention that $\epsilon \in (0,1)$ with the optimal choice being in close proximity to $1$. This choice results in ${\gamma}$ being close to $1/(2(1+4 \widetilde{C}/C_V))$. 

In contrast, when dealing with a confinement potential of smoothness $J$, we encounter the constraint $m \le J$. In this case, the optimal choice is to set $m = J$, leading to ${\gamma}$ being close to $J/(2(J+2)(1+4 \widetilde{C}/C_V))$.
    \qed
\end{remark}

{\rev \subsection{A Lower Bound} \label{secLowe}

In this section, we establish a lower bound for our statistical problem in the non-smooth setting. To this end, we consider a simplified statistical model described as follows:
\[
X_1,\ldots, X_N \text{ are  i.i.d. } \sim \pi, \qquad  \pi (x) = 
\frac{1}{Z_{\pi}} \exp \left( - 2V(x) - W \star \pi (x) \right).
\]
We denote the corresponding law by $\mathbb P^{\otimes N}_{V,W}$. 
We assume that the confinement potential $V$ is fixed and takes the form $V(x)=\alpha x^2/2 +\widetilde{V}(x)$, where $\alpha>0$ is a known constant, and $\widetilde{V}$ satisfies Assumption A2. Next, we introduce the functional class $\mathcal{A}_{r,p,J}$, $r=(r_1,\ldots, r_6)\in \R_{>0}^6$ 
that satisfies \cref{ass: beta} with $\|W'\|_{\infty}<\sqrt{C_V-C_W}$ and further
\begin{align*}
&\inf_{x\in \mathbb{R}} W''(x) \in [-r_1,-r_2], \qquad \limsup_{x\to \infty}  \exp(2px^2) \int_{|y|>x} |W'(y)|^2 dy \leq r_3, \\[1.5 ex]
&\int_{\mathbb{R}} \left| \frac{\mathcal{F}(W')(z)}{\mathcal{F}(\pi)(z)} \right|^2 dz  \leq r_4,
\qquad r_5(1\wedge |z|^{-J-2})\leq |\mathcal{F}(\pi)(z)| \leq r_6(1\wedge |z|^{-J-2}),
\end{align*} 
for $p>0$. We remark that the space $\mathcal{A}_{r,p,J}$ is non-empty. In particular, existence of confinement potential $V$,  which guarantees the validity of the last condition on $\mathcal{F}(\pi)$, is shown in Example  
\ref{exmp}. The main result of this section is the following statement.

\begin{theorem} \label{ThLowe}
For some $(p,r)\in \R_{>0} \times \R_{>0}^6$ there exists a constant $c_0>0$ such that
\[
\inf_{W'_N} \sup_{W\in \mathcal{A}_{r,p,J}} \mathbb P^{\otimes N}_{V,W} \left(
\|W'_N - W'\|_{L^2(\R)} >c_0 N^{-1/4}\right) >0,
\]
where the infimum is taken over all estimators $W'_N$.  
\end{theorem}

\noindent
We recall that the convergence rate of our estimator introduced in the previous section is given by $N^{-\gamma/2}$ (up to log terms) with 
\[
\gamma= \frac{J}{2(J+2)(1+4\widetilde{C}/C_V)}.
\]
While the latter does not match the lower bound of Theorem \ref{ThLowe}, the rate $N^{-\gamma/2}$ is getting close to $N^{-1/4}$ for large $J$ and $C_V$. }

\section{Proof of the Main Results}{\label{S: proof main}}
Let us introduce some notation and properties of the Laplace transform, as defined in \eqref{eq: def laplace}, that will be useful in the following section. 
The inverse Laplace transform is given by the following formula 
$$\phi(t) = \frac{1}{2 \pi \imath} \int_{x - \imath \infty}^{x + \imath \infty} \widehat{\phi} (z) \exp(zt) dz = \frac{1}{2 \pi} \int_{- \infty}^{\infty} \widehat{\phi}(x + \imath y)\exp((x + \imath y)t) dy,$$
with $x \in (x_\phi^-, x_\phi^+)$. The bilateral Laplace transform is unique in the sense that if $\widehat{\phi}_1$ and $\widehat{\phi}_2$ are such that $\widehat{\phi}_1 (z) = \widehat{\phi}_2 (z)$ in a common strip of convergence $\operatorname{Re}(z) \in (x_{\phi_1}^-, x_{\phi_1}^+) \cap (x_{\phi_2}^-, x_{\phi_2}^+)$, then $\phi_1(t) = \phi_2(t)$ for almost all $t\in\R$ (see Theorem 6b in \cite{Wid46}). \\
\\
Moreover, the following identities will prove to be useful. Let $\phi$ be a locally integrable function and $a>0$. By Parseval's theorem, (see Theorem 31.7 in \cite{Doe})
\begin{align*}
\| \phi \|_2^2 
&= \frac{1}{2\pi \imath} \int_{a-\imath\infty}^{a+\imath\infty} 
\overline{\mathcal{F}(\phi)(-\imath \overline s)} \mathcal{F} (\phi)(\imath s) d s\\
&= \frac{1}{2\pi} \int_{-\infty}^\infty \overline{\mathcal{F}(\phi)(- \imath a - y) } \mathcal{F} (\phi) (\imath a - y) d y.
\end{align*}
Then, 
\begin{align*}
\| \phi \|^2_2 &\lesssim \int_{-\infty}^\infty | \overline{\mathcal{F}(\phi)(- \imath a - y) } \mathcal{F} (\phi) (\imath a - y) | d y \lesssim \Delta_- +\Delta_+,
\end{align*}
where 
\begin{align*}
\Delta_- &= \int_{-\infty}^\infty  | \overline{\mathcal{F}(\phi)(- \imath a - y) } |^2 d y = \int_{\mathcal{L}_{-a}} |\mathcal{F}(\phi)(z)|^2 d z,\\ 
\Delta_+ &= \int_{-\infty}^\infty |\mathcal{F}(\phi)(\imath a-y)|^2 dy = \int_{\mathcal{L}_a} | \mathcal{F} (\phi) (z) |^2 d z,
\end{align*}
and $\mathcal{L}_a = \{y + \imath a: y \in \mathbb{R} \}$. \\

\subsection{Proof of Theorem \ref{th: dual formulation}}
As it will be helpful for the proof of Theorem \ref{th: dual formulation}, we now state a lemma establishing that finite exponential moments of $\overline X^1_0=X^{1,N}_0=X^i_0$ imply finite exponential moments of $\overline X^1_t$ and $X^{1,N}_t$ for all $t \ge 0$. Its proof can be found in Section \ref{Sec6.3}. \\
Remark 
that the sub-Gaussianity implies the required
finite exponential moments \eqref{eq: cond exp moments}. Indeed, a centered random variable $Z$ is sub-Gaussian if there exists a $K > 0$ such that $\E[\exp(t Z)] \le \exp(K^2 t^2/2)$ for all $t \in \mathbb{R}$. For sufficient conditions ensuring sub-Gaussianity see Definition 30 in \cite{MaeHof} and references therein.

\begin{lemma}{\label{l: moments}}
Under \cref{ass: beta}, we have 
$$\sup_{t \ge 0}\mathbb{E} \left[ \exp (\pm c \overline X_t) + \exp(\pm c X^{i,N}_t) \right] < \infty$$
uniformly in $i$ and $N$. Additionally, for any $k \ge 1$,
\begin{equation}{\label{eq: bounded moments}}
\sup_{t \ge 0}\mathbb{E} \left[ |\overline X_t|^k + |X^{i,N}_t|^k \right] < \infty
\end{equation}
uniformly in $i$ and $N$.
\end{lemma}

\noindent

\noindent
We start proving Theorem \ref{th: dual formulation} by introducing some notation. Let $\overline{\Pi}_t$ be the law of $\overline{X}_t$, $\Pi$ the law with density $\pi$ and $\Pi^N_T=\Pi_{N,T}$ with 
 \begin{align*}
\mathcal{F}(\Pi^N_T)(z) &= \frac{1}{N} \sum_{j=1}^N \exp(\imath z X_T^{j,N}), \qquad \mathcal{F}(\overline \Pi^N_T)(z) = \frac{1}{N} \sum_{j=1}^N \exp(\imath z \overline X_T^j),\\
\mathcal{F}(\overline \Pi_T)(z) &= \mathbb{E} [ \exp(\imath z \overline X_T) ], \qquad \mathcal{F}(\Pi) (z) = \mathcal{F}(\pi) (z) = \int_{-\infty}^\infty \exp(\imath z x) \pi (x) d x. 
\end{align*}
In order to prove the result we now consider the following decomposition:  
$$
\mathbb{E} [  | \mathcal{F}(\Pi) (z) - \mathcal{F}(\Pi^N_T) (z) |^2 ] \lesssim \sum_{i=1}^3 \delta_i(z), 
$$
where
\begin{align*}
 \delta_1(z) = | \mathcal{F}(\Pi) (z) - \mathcal{F}(\overline \Pi_T) (z) |^2,\qquad
  \delta_2(z) = \mathbb{E} \left[ |\mathcal{F}(\overline \Pi_T) (z) - \mathcal{F}(\overline \Pi^N_T) (z) |^2 \right],
\end{align*}
and
\begin{align*}
   \delta_3(z) &= \mathbb{E} \left[ |\mathcal{F}(\overline \Pi^N_T) (z) - \mathcal{F}(\Pi^N_T) (z) |^2 \right]\\
&\le \left( \frac{1}{N} \sum_{j=1}^N \mathbb{E} \left[ |\exp(\imath z\overline X_T^j) - \exp(\imath z X_T^{j,N}) |^2\right]^{\frac{1}{2}} \right)^2\\ 
&= \mathbb{E} \left[ |\exp(\imath z \overline X^1_T) - \exp(\imath z X_T^{1,N}) |^2 \right].
\end{align*}
For $x = x_1 + \imath x_2$ and $y = y_1 + \imath y_2$, we have that
\begin{align}{\label{eq: bound dual}}
\left| \exp(x+y) - \exp(y) \right| 
&=  \left| \exp(x) - 1 \right| \exp(y_1) \nonumber \\
&\le \left( \left| \exp(\imath x_2) - 1  \right| \exp(x_1)  + \left| \exp(x_1) - 1 \right| \right) \exp(y_1) \nonumber\\
&\le (|x_2| \exp(x_1) + |x_1| (\exp(x_1) + 1)) \exp(y_1)\nonumber \\ 
&\le \sqrt{2} |x| (\exp(x_1+y_1) + \exp(y_1)) 
\end{align}
which implies, for any $z \in \mathcal{L}_a$,
$$
\delta_3(z) \le 2 |z|^2 \mathbb{E} \left[ | \overline X^1_T- X^{1,N}_T |^2 (\exp(- a \overline X^1_T) + \exp(-a X^{1,N}_T) )^2 \right].
$$
Using the Cauchy-Schwarz inequality, we get
$$
\delta_3(z) \lesssim |z|^2 \mathbb{E} \left[ | \overline X^1_T- X^{1,N}_T |^4\right]^\frac{1}{2},
$$
because exponential moments of $\overline X^1_T$ and $X^{1,N}_T$ are uniformly in $T,N$ finite according to Lemma \ref{l: moments}. By Proposition \ref{prop: prop chaos L4},
$$\mathbb{E} \left[ | \overline X^1_T- X^{1,N}_T |^4 \right] \lesssim \frac{1}{N^2},$$
which yields
\begin{equation}{\label{eq: delta3}}
\delta_3(z) \lesssim \frac{|z|^2}{N}.
\end{equation}
Next, consider $\delta_2(z)$. Since $\overline \Pi^N_T$ is based on i.i.d.\ $\overline X^1_T, \dots, \overline X^N_T$ with common law $\overline \Pi_T$, we have that
\begin{equation}{\label{eq: delta2}}
\delta_2(z) = \frac{1}{N} \mathbb{E} \left[ | \mathbb{E} [\exp(\imath z \overline X^1_T)]-\exp(\imath z \overline X^1_T) |^2 \right] 
\lesssim \frac{1}{N},  
\end{equation}
where exponential moments of $\overline X^1_T$ are uniformly in $T$ finite 
by Lemma \ref{l: moments}.
Finally, let us deal with $\delta_1(z)$. We introduce a random vector $(\overline X_T,X)$ with $\mathcal{L}(\overline X_T) = \overline \Pi_T$, $\mathcal{L}(X) = \Pi$ which attains $\mathbb{E}[|\overline X_T - X|^2] = W_2^2(\overline \Pi_T, \Pi)$. Then, using again \eqref{eq: bound dual}, it is
\begin{align*}
\delta_1 (z)^\frac{1}{2} &= |\mathbb{E}[\exp(\imath z \overline X_T)] - \mathbb{E} [\exp(\imath zX) ] |\\ 
&\le \mathbb{E} [ |\exp(\imath z \overline X_T) - \exp(\imath z X) | ] \\
& \le \sqrt{2} |z| \mathbb{E} [ |\overline X_T - X | 
(\exp(-a\overline X_T) + \exp(-aX)) ]
\end{align*}
Using the Cauchy-Schwarz inequality we obtain 
\begin{align*}
&\mathbb{E} [ |\overline X_T - \overline X |  (\exp(-a\overline X_T) + \exp(-a X) ) ] 
\lesssim W_2(\overline \Pi_T, \Pi) 
\end{align*}
since the moments and the exponential moments of $\overline X_T$, $X$ are uniformly in $T$ finite by Lemma \ref{l: moments} and the fact that $\mathcal{L}(X)=\Pi$ is the invariant law of the McKean-Vlasov stochastic differential equation.
Theorem 1.4 in \cite{Mal03} ensures that there exist two constants $c$ and $\lambda := C_V-C_W > 0$ such that 
$$W_2(\overline \Pi_T, \Pi) \le c \exp(- \lambda T).$$
It yields that
\begin{equation}{\label{eq: delta1}}
\delta_1(z) \lesssim |z|^2 \exp(- 2\lambda T). 
\end{equation}
From \eqref{eq: delta3}, \eqref{eq: delta2} and \eqref{eq: delta1} we get, for any $z \in \mathcal{L}_a$,
$$\mathbb{E} [ |\mathcal{F}(\Pi) (z) - \mathcal{F}(\Pi^N_T) (z) |^2 ] \lesssim |z|^2\left(\frac{1}{N} + \exp(- \lambda T)\right) = \frac{|z|^2}{N_T}. $$
Our reasoning easily extends to $z \in \mathcal{L}_{-a}$, thereby concluding the proof of the theorem.

\subsection{Proof of Proposition \ref{prop: bound l}}
The proof of Proposition \ref{prop: bound l} follows closely the proof of Proposition 4.1 in \cite{BPP}. 
Recall that $l_{N,T}(x)$ is a kernel estimator of 
\begin{equation}\label{eq:lrepr}
l(x) = \frac{\pi'(x)}{\pi (x)} = -2 V'(x) - W' \star \pi (x).
\end{equation}
Let us decompose its error into the sum 
\begin{equation}\label{ineq:ldiff}
|l_{N,T}(x)-l(x)| = |l(x)| \one_{\{\pi_{N,T}(x) \le \delta\}} + r(x)    
\end{equation}
where
\begin{align}
r(x) &:= | l_{N,T} (x)  - l (x) | \one_{\{\pi_{N,T}(x) > \delta\}}\nonumber\\
&\le \delta^{-1}\left( |\pi'_{N,T}(x)-\pi'(x)| +  |l(x)| |\pi_{N,T}(x) - \pi(x)|\right). \label{ineq:ldiff_rmd}
\end{align}
Under our assumptions, we have that
$|V'(x)| \le c (1+|x| + |\widetilde{V}'(x)|) \le c(1 + |x|)$, where we have used that $\widetilde{V}'(0) = 0$ and $\widetilde{V}''$ is bounded. Moreover, 
$| W' \star \pi (x) | \le \| W' \|_\infty \| \pi \|_1 = \| W' \|_\infty < \infty$, which imply
$$
\sup_{|x| \le U} |l(x)| \lesssim U.
$$
Next, Lemma \ref{l: bound pi} gives the lower bound $\pi(x) \ge c_1 \exp (- \widetilde C |x|^2)$, which in turn implies $\pi(x) \ge 2 \delta$ for all $|x| \le U$. It follows that for all $|x| \le U$,
\begin{align*}
\mathbb{P}(\pi_{N,T}(x) \le \delta) &= \mathbb{P} (\pi(x) - \pi_{N,T}(x) \ge \pi(x) - \delta)\\
&\le \mathbb{P} (\| \pi - \pi_{N,T} \|_\infty \ge \delta ) \le \delta^{-2} \mathbb{E} \left[\| \pi - \pi_{N,T} \|^2_\infty \right].
\end{align*}
Finally, consider
$$
\pi_{N,T} (x) - \pi(x) = r_0 (x) + r_1(x)
$$
with
$$
r_0(x) = K_{h_0} \star (\Pi_{N,T}(x) - \Pi) (x), \quad r_1(x) = K_{h_0} \star \Pi (x) - \pi (x),
$$
where recall $K_{h_0} (x) = h_0^{-1} K(h_0^{-1} x)$ is a scaled kernel. We get
$$
|r_0(x)| \le c h_0^{-1} W_1 (\Pi_{N,T},\Pi)
$$ 
by applying the Kantorovich-Rubinstein theorem, moreover, $\E[W_1^2(\Pi_{N,T}, \Pi)] \le c N_T^{-1}$ because of Theorem \ref{th: prop chaos}. After substitution, we have that
$$
r_1(x) = \int_\R (\pi(x+h_0y) - \pi (x)) K(y) dy,
$$
where by the Taylor theorem 
$$
\pi(x+yh_0) = \pi (x) + \sum_{i=1}^{m-1} \frac{\pi^{(k)}(x)}{k!} (yh_0)^k + \frac{\pi^{(m)}(x+\tau yh_0 )}{m!} (yh_0)^m
$$
for some $0 \le \tau \le 1$. Recall that $K$ has order $m$. Moreover, the bound in Lemma \ref{l: bound pi} ensures that $|\pi^{(m)} (x)| \le c$, hence,
$$
|r_1(x)| \le c h_0^m
$$
uniformly in $x \in \mathbb{R}$. Our choice of $h_0$ yields
$$
\mathbb{E} \left[\| \pi_{N,T} -\pi \|_\infty^2 \right] \lesssim N_T^{-\frac{m}{m+1}}
$$
and similarly, that of $h_1$ yields
$$
\mathbb{E}\left[\| \pi'_{N,T} - \pi'  \|_\infty^2\right] \lesssim N_T^{-\frac{m}{m+2}}.
$$
Using these bounds in \eqref{ineq:ldiff}, \eqref{ineq:ldiff_rmd}, we obtain
$$
\sup_{|x| \le U} \mathbb{E}\left[|l_{N,T}(x)-l(x)|^2\right]^{\frac{1}{2}} \lesssim \exp (\widetilde C U^2) \left(N_T^{-\frac{m}{2(m+2)}} + U N_T^{-\frac{m}{2(m+1)}}\right),
$$
which concludes the proof.
\subsection{Proof of Theorem \ref{th: estim psi}}
Recall that
\begin{align*}
\Psi_{N,T}(x) &= \left(l_{N,T}(x) +  2 \alpha_{N,T} x + 2 \widetilde{V}'(x)\right) 1_{\{|x| \le \epsilon U\}},\\
\Psi(x) &= l(x) +  2 \alpha x + 2 \widetilde{V}^\prime(x)= - W' \star \pi (x).
\end{align*}
We decompose the mean integrated squared error of $\Psi_{N,T}$ into  
\begin{equation}{\label{eq: bound psi 14.5}}
{\rev \left( \mathbb{E} \left[ \int_{\mathbb{R}} \left(\Psi_{N,T}(x)-\Psi(x)\right)^2 dx \right] \right)^{\frac{1}{2}}} = \left( \int_{|x| > \epsilon U} (\Psi (x))^2 d x \right)^{\frac{1}{2}} + I
\end{equation}
where
$$
I:= \left( \int_{|x| \le \epsilon U} \mathbb{E} \left[ (\Psi_{ N,T}(x)-\Psi(x))^2 \right] d x \right)^{\frac{1}{2}}.
$$
Applying Minkowski's inequality, we get $I\le I_1 + I_2$, where
\begin{align*}
I_1 &:= \left( \int_{|x| \le \epsilon U} \mathbb{E}\left[(l_{N,T}(x) - l (x))^2\right] dx \right)^{\frac{1}{2}} 
\lesssim U^{\frac{1}{2}} \sup_{|x|\le U} \left( \mathbb{E}\left[(l_{N,T}(x)-l(x))^2\right] \right)^{\frac{1}{2}},\\
I_2 &:= \left( \int_{|x| \le \epsilon U} \mathbb{E}\left[ (2 \alpha_{N,T}x - 2 \alpha x)^2 \right] dx \right)^{\frac{1}{2}} 
\lesssim U^{\frac{3}{2}}  {\rev \left( \mathbb{E}  \left[(\alpha_{N,T}-\alpha)^2\right] \right)^{\frac{1}{2}}. }
\end{align*}
Next, consider the mean squared error of
$$
\alpha_{N,T} := \operatorname{arg} \min_{\alpha \in \mathbb{R}} \int_{\mathbb{R}} \left(l_{N,T}(x) + 2 \alpha x + 2 \widetilde{V}'(x)\right)^2 w_U(x) d x
$$ 
where {\rev we recall that} $w_U (\cdot) = w(\cdot/U)/U$. 
The estimator $\alpha_{N,T}$ and $\alpha$ can be computed explicitly via
\begin{align*}
\alpha_{N,T} &= - \frac{1}{2 C_2 U^2} \int_{\mathbb{R}} \left(l_{N,T}(x) + 2\widetilde{V}'(x)\right) x w_U (x) d x,\\ 
\alpha  &= - \frac{1}{2 C_2 U^2} \int_{\mathbb{R}} \left(l(x)+ 2\widetilde{V}'(x) + W' \star \pi (x)\right) x w_U(x) d x,
\end{align*} 
where $C_2 := \int_{\mathbb{R}} x^2 w (x) d x$. Since the support of $w_U$ is $[\epsilon U, U]$, Jensen's inequality can be applied to get
\begin{align*}
(\alpha_{N,T}-\alpha)^2 
\le \frac{1}{4 C^2_2 U^4} \int_{\mathbb{R}} \left(l_{N,T}(x)-l(x) - W' \star \pi (x)\right)^2 x^2 w_U(x) d x.
\end{align*}
By Minkowski's inequality, we obtain
$$
{\rev \left( \mathbb{E}  \left[(\alpha_{N,T}-\alpha)^2\right] \right)^{\frac{1}{2}}} \le \frac{1}{2C_2U^2}(J_1 + J_2),
$$
where
\begin{align*}
J_1 &:= \left( \int_{\mathbb{R}} \mathbb{E}\left[(l_{N,T}(x)-l(x))^2\right] x^2 w_U (x) d x \right)^{\frac{1}{2}}
\le C_2^{\frac{1}{2}} U \sup_{|x| \le U} {\rev \left( \mathbb{E}\left[(l_{N,T}(x)-l(x))^2\right] \right)^{\frac{1}{2}}},\\
J_2 &:= \left( \int_{\mathbb{R}} (W' \star \pi(x))^2 x^2 w_U (x) d x \right)^{\frac{1}{2}} \le \frac{C_{\infty} U^{\frac{1}{2}} }{2} 
\left( \int_{|x| > \epsilon U} (W' \star \pi (x) )^2 d x \right)^{\frac{1}{2}}
\end{align*}
with $C_{\infty} := \sup_{x \in \mathbb{R}} x^2 w(x).
$
Using the above bounds in \eqref{eq: bound psi 14.5} we conclude that 
\begin{align*}
{\rev \left( \mathbb{E} \left[ \int_{\mathbb{R}} (\Psi_{N,T}(x)-\Psi(x))^2 dx \right] \right)^{\frac{1}{2}} }
&\lesssim U^{\frac{1}{2}} \sup_{|x|\le U} {\rev \left( \mathbb{E}\left[(l_{N,T}(x)-l(x))^2\right] \right)^{\frac{1}{2}}}\\ 
&+ \left( \int_{|x| > \epsilon U} ( W' \star \pi(x) )^2 dx \right)^{\frac{1}{2}}.    
\end{align*}
The first term on the right hand side has been studied in Proposition \ref{prop: bound l}, and we are therefore left to study the second term. It is 
$$\left( \int_{|x| > \epsilon U} ( W' \star \pi (x) )^2 dx \right)^{\frac{1}{2}} = \left( \int_{|x| > \epsilon U} \left( \int_{\mathbb{R}} W'(x-y) \pi(y) dy \right)^2 dx \right)^{\frac{1}{2}}.$$
By Minkowski's inequality this is bounded by
\begin{multline}{\label{eq: bound conv 1}}
 \int_{|y| \le \frac{\epsilon}{2}U}  \left( \int_{|x|>\epsilon U} |W'(x-y)|^2 d x \right)^{\frac{1}{2}} \pi (y) d y 
 \\
 + \int_{|y| > \frac{\epsilon}{2}U} \left( \int_{|x|>\epsilon U} |W'(x-y)|^2 d x \right)^{\frac{1}{2}} \pi (y) dy.
\end{multline}
Then, we apply a change of variables $x-y := \widetilde{x}$, observing that $|x| > \epsilon U$ and $|y| \le \epsilon U/2$ imply $|x - y| > \epsilon U/2$. For the second integral we enlarge the domain of integration to $\R$. It follows that \eqref{eq: bound conv 1} is upper bounded by 
$$\left( \int_{|x|>\frac{\epsilon}{2} U} |W'(x)|^2 d x \right)^{\frac{1}{2}} \int_{\R} \pi(y) dy + \left( \int_{|y|>\frac{\epsilon}{2} U} \pi(y) d y \right) \left( \int_{\R} |W'(x)|^2 d x \right)^{\frac{1}{2}}.$$
Thanks to Lemma \ref{l: bound pi}, we know that
$$\int_{\frac{\epsilon}{2}U}^\infty \pi(x) dx \lesssim \int_{\frac{\epsilon}{2}U}^\infty \overline{\pi}(x) dx, $$
with $\overline{\pi}(x) = c_2 \exp(- C_V x^2)$ satisfying $\overline{\pi}'(x) = - 2 c_2 C_V x \overline{\pi}(x)$. Hence, we can write 
$$\int_{u}^\infty \overline{\pi}(x) dx \le \frac{1}{u} \int_{u}^\infty x \overline{\pi}(x) dx = \frac{\overline{\pi}(u)}{2 c_2 C_V u}. $$
It implies that
\begin{equation}{\label{eq: bound pi integral}}
\int_{\epsilon \frac{U}{2}}^\infty \pi(x) dx \lesssim \frac{ 2\exp(- C_V (\frac{\epsilon}{2}U)^2)}{ \epsilon U} 
\end{equation}
Then, the boundedness of $\int_{\R} \pi(y) dy$ is straightforward, while $\int_{\R} |W'(x)|^2 dx $ is bounded as $W' \in L^1(\R) \cap L^\infty(\R)$. 
It follows 
\begin{equation}{\label{eq: bound conv 2}}
\left( \int_{|x| > \epsilon U} ( W' \star \pi (x) )^2 dx \right)^{\frac{1}{2}} \lesssim \left( \int_{|x|>\frac{\epsilon}{2} U} |W'(x)|^2 d x \right)^{\frac{1}{2}} + \frac{2 \exp(- C_V(\frac{\epsilon}{2}U)^2)}{ \epsilon U},
\end{equation}
as we wanted.

\subsection{Proof of Corollary \ref{cor: rate psi}}
 Corollary \ref{cor: rate psi} is a consequence of Theorem \ref{th: estim psi} and of the exponential decay of the tails of $W'$. The choice of the threshold
\begin{equation}{\label{eq: choice U 9.75}}
U^2=c_u \log(N_T)
\end{equation}
gives
$$
\left( \int_{|y|>\frac{\epsilon}{2} U} |W'(y)|^2 d y \right)^{\frac{1}{2}} \lesssim \exp \left(-  \frac{p\epsilon^2U^2}{4}  \right) = N_T^{-\frac{p\epsilon^2 c_u}{4}}.
$$
Together with Theorem \ref{th: estim psi} it implies that 
\begin{align*}
{\rev \left( \mathbb{E} \left[ \int_{\mathbb{R}} (\Psi_{N,T}(x)-\Psi(x))^2 dx \right] \right)^{\frac{1}{2}} }
&\lesssim 
    (\log N_T)^{\frac14}N_T^{\widetilde{C} c_u} \left(N_T^{- \frac{m}{2(m + 2)}} +  (\log N_T)^{\frac{1}{2}} N_T^{- \frac{m}{2(m + 1)}}\right) \\
    & + N_T^{- C_V \frac{\epsilon^2}4 c_u} + N_T^{-p\frac{\epsilon^2}4c_u}.
 \end{align*}
Recall that $p > C_V $. 
Then, we can choose $c_u$ in order to obtain the balance between the remaining two terms above: 
\begin{equation}{\label{eq: choice cu}}
 c_u = \frac{m}{2(m + 2)} \frac{1}{(\widetilde{C} + C_V \frac{\epsilon^2}{4})}. 
\end{equation} 
Then, the convergence rate is 
\begin{align*}
&(\log N_T)^{\frac14} N_T^{- \frac{m}{2(m+2)} \frac{C_V \frac{\epsilon^2}{4}}{(\widetilde C + C_V \frac{\epsilon^2}{4})}} 
\end{align*} 
as claimed.


\subsection{Proof of Theorem \ref{th: estim beta}}

Let $\mathcal{L}_a:=\{y+\imath a: y \in \mathbb{R}\}$ for some $a \ge 0$. Assume $\operatorname{Re}(\mathcal{F}(\Pi)(s))>0$ for all $s \in \mathcal{L}_a$. Recall that we defined a sequence of entire functions $\rho_{N, T}(s) := \mathcal{F}(\Pi)(s)-\mathcal{F}(\Pi_{N,T})(s) + \varepsilon_{N,T}$, $s \in \mathbb{C}$, for some $\varepsilon_{N, T}>0$. Note that for all $s \in \mathcal{L}_a$
\begin{equation}\label{ineq:vareps}
|\mathcal{F}(\Pi_{N,T})(s) + \rho_{N,T}(s)| =  |\mathcal{F}(\Pi)(s) + \varepsilon_{N,T}| \ge \varepsilon_{N,T} > 0
\end{equation}
and $W'_{N, T}$ is defined via
$$
\mathcal{F} (W'_{N, T} )(s):=-\frac{\mathcal{F}(\Psi_{N, T})(s)}{\mathcal{F}(\Pi_{N, T})(s)+\rho_{N, T}(s)}.
$$
Let $\delta_{N,T} := W'_{N,T}-W'$. Plancherel's theorem gives
\begin{align}\label{ineq:intd}
\int_{-\infty}^\infty | \delta_{N,T} (x)|^2 dx &= \frac{1}{2\pi \imath} \int_{a-\imath\infty}^{a+\imath\infty} \overline{\mathcal{F}(\delta_{N,T})(\imath(-\overline{s}))} \mathcal{F}(\delta_{N,T})(\imath s) ds\nonumber\\ 
&= \frac{1}{2\pi} \int_{-\infty}^\infty \overline{\mathcal{F}(\delta_{N,T})(y-\imath a)} \mathcal{F}(\delta_{N,T})(y + \imath a)dy \lesssim \Delta_{-a} + \Delta_a,
\end{align}
where
\begin{align*}
\Delta_{\pm a} := \int_{-\infty}^\infty |\mathcal{F}(\delta_{N,T})(y \pm \imath a)|^2 d y =  \int_{\mathcal{L}_{\pm a}} |\mathcal{F}(\delta_{N,T})(s)|^2 d s.
\end{align*}
It suffices to consider $\Delta_a$, because the analysis of $\Delta_{- a}$ follows a similar route. Rewrite
$$
\mathcal{F}(\delta_{N,T})(s) = -\frac{\mathcal{F}(\Psi_{N, T})(s)- \mathcal{F}(\Psi)(s) + \varepsilon_{N,T} \mathcal{F}(W')(s)}{\mathcal{F}(\Pi_{N, T})(s)+\rho_{N, T}(s)}.
$$
Let us deal with the denominator by using \eqref{ineq:vareps} and $\operatorname{Re}(\mathcal{F}(\Pi)(s))+ \varepsilon_{N,T} > \operatorname{Re}(\mathcal{F}(\Pi)(s)) > 0$ for all $s \in \mathcal{L}_a$. We get
\begin{equation}\label{ineq:Delta}
\Delta_a \lesssim \frac{1}{\varepsilon_{N,T}^2} \Delta_{a,1} + \varepsilon_{N,T}^2 \int_{\mathcal{L}_a} \left|\frac{\mathcal{F}(W')(s)}{\mathcal{F}(\Pi)(s)} \right|^2 ds,
\end{equation}
where 
$$
\Delta_{a,1} := 
\int_{\mathcal{L}_a} |\mathcal{F}(\Psi_{N, T})(s) - \mathcal{F}(\Psi)(s)|^2 d s.
$$

\begin{align}{\label{eq: start deltaa1 0.5}}
{\rev \left(\mathbb{E} [ \Delta_{a,1} ] \right)^{\frac 1 2}} = 
\left( \int_{-\infty}^\infty \exp(-2 a x) \mathbb{E} \left[ |\Psi_{N,T}(x)-\Psi(x)|^2 \right] d x \right)^{\frac 1 2} \le 
D_1 + D_2,
\end{align}
where
\begin{align*} 
D_1 &:= \left( \int_{|x|> \epsilon U} \exp(-2ax) |\Psi (x)|^2 d x \right)^{\frac 1 2},\\
D_2 &:=  (2 \epsilon U)^{\frac 1 2} \exp(a\epsilon U) \sup_{|x| \le \epsilon U} {\rev \left( \mathbb{E} \left[ \int_{\mathbb{R}} (\Psi_{N,T}(x)-\Psi(x))^2 dx \right] \right)^{\frac{1}{2}} }.
\end{align*}
We start handling $D_1$, while the analysis on $D_2$ heavily relies on the bounds gathered in previous steps. We recall that
$$
\Psi(x) = - W' \star \pi (x) = - \int_{\mathbb{R}} W'(x-y) \pi (y) d y.
$$
By the Minkowski inequality,
\begin{align*}
D_1 &\le \int_{\mathbb{R}} \left( \int_{|x| > \epsilon U} \exp(-2ax) W'(x-y)^2 d x \right)^{\frac{1}{2}} \pi(y) dy
= I_1 + I_2,
\end{align*}
where after a change of variable the r.h.s.\ has been decomposed into
\begin{align*}
I_1 &:= \int_{|y+c_0| \le \frac{\epsilon}{2}U} \left( \int_{|x+y|>\epsilon U} \exp (-2ax) W'(x)^2 d x \right)^{\frac{1}{2}} \exp(-ay) \pi (y) d y,\\ 
I_2 &:= \int_{|y+c_0| > \frac{\epsilon}{2}U} \left( \int_{|x+y|>\epsilon U} \exp(-2ax) W'(x)^2 d x \right)^{\frac{1}{2}} \exp(-ay) \pi (y) d y
\end{align*}
with $c_0 := \frac{a}{2C_V}$.
Note that $|y+c_0| \le \frac{\epsilon}{2} U$ and $|x+y|>\epsilon U$ imply $|x-c_0|> \frac{\epsilon}{2}U$ in the inner integral in $I_1$. Let us also enlarge the domain of integration to $\mathbb{R}$ in the outer and inner integrals in $I_1$ and $I_2$ respectively. Then 
\begin{equation}\label{ineq:D23}
I_1+I_2 \le J_1 \int_{\mathbb{R}} \exp(-ay) \pi (y) dy +  J_2 \left( \int_{\mathbb{R}} \exp (-2ax) W'(x)^2 d x \right)^{\frac{1}{2}}, 
\end{equation}
where 
$$
J_1 := \left( \int_{|x - c_0|> \frac{\epsilon}{2} U} \exp(-2ax) W'(x)^2 d x \right)^{\frac{1}{2}}, \qquad J_2 := \int_{|y + c_0|>\frac{\epsilon}{2}U} \exp(-ay) \pi(y) dy.    
$$
The upper bound on $\pi$ in Lemma \ref{l: bound pi} implies that the first integral on the r.h.s.\ of \eqref{ineq:D23} is finite,
furthermore, 
\begin{align*}
J_2 &\lesssim \int_{|y+c_0|> \frac{\epsilon}{2}U} \exp(-ay-C_V y^2) d y\\
&\eqsim  \int_{|y|> \frac{\epsilon}{2}U} \exp(-C_V y^2) d y \lesssim \frac{\exp(-C_V (\frac{\epsilon}{2}U)^2)}{\frac{\epsilon}{2}U}.
\end{align*}
Now consider $J_1$ and the second integral on the r.h.s.\ of \eqref{ineq:D23}, where, recall, $W' \in L^2 (\mathbb{R})$ is odd and satisfies the assumption \eqref{eq: cond tails beta} for $p > C_V$.
This means that there exists $c>0$ such that 
$$
\int_x^\infty W'(u)^2 d u \le c \exp(-2px^2)
$$ 
for all $x \ge 1$, that is 
$$
\int_1^\infty \chi (u) W'(u)^2 d u \le c \int_1^\infty \chi (u) d F (u), 
$$
where $F (u) := 1- \exp (- 2p u^2)$ and $\chi (u) := \one (u\ge x)$, $u \in \mathbb{R}$, for all $x \ge 1$. By monotone approximation the above inequality remains valid for all non-negative non-decreasing functions $\chi : [1,\infty) \to [0,\infty)$, for example, $\chi(u) = \exp (2au)$. We get that the second integral of $W' \in L^2(\mathbb{R})$ on the r.h.s.\ of \eqref{ineq:D23} is finite, 
moreover, 
\begin{align*}
(J_1)^2 &\lesssim \int^{\infty}_{\frac{\epsilon}{2}U-c_0} \exp (2ax) W'(x)^2 dx\\ &\lesssim \int_{\frac{\epsilon}{2}U-c_0}^\infty \exp (2ax) d F(x) = \int_{\frac{\epsilon}{2}U-c_0}^\infty \exp (2ax-2px^2)(4px) dx\\
&\lesssim \exp \left(-2p \left( \frac{\epsilon}{2}U-c_1 \right)^2 \right),
\end{align*}
where $c_1=c_0+\frac{a}{2p}$.
Since $p > C_V$, it follows that 
$$
(J_1)^2 \lesssim \frac{\exp(-2C_V (\frac{\epsilon}{2}U)^2)}{(\frac{\epsilon}{2}U)^2}.
$$
We conclude that 
\begin{equation}\label{ineq:D1conv}
D_1 \lesssim \frac{\exp(-C_V (\frac{\epsilon}{2}U)^2)}{\frac{\epsilon}{2}U}.
\end{equation}

Regarding $D_2$, we recall that from the definition of $\Psi_{N,T}(y)$ and $\Psi(y)$ we deduce
$$
D_2 \lesssim (2\epsilon U)^{\frac{1}{2}} \exp (a \epsilon U) (R_1 + R_2),
$$
where 
\begin{align*}
R_2 &:= \sup_{|y| \le U} {\rev \left( \E\left[|l_{N,T}(y) - l(y)|^2\right] \right)^\frac{1}{2} },\\
R_1 &:= \epsilon U {\rev \left( \E\left[\left( \alpha_{N,T} - \alpha \right)^2\right] \right)^\frac{1}{2}} \lesssim R_2 + U^{-\frac{1}{2}} \left( \int_{|x|>\epsilon U} \Psi (x)^2 dx \right)^{\frac{1}{2}}.
\end{align*}
The integral in the upper bound on $R_1$ coincides with $D_1$ when $a=0$ hence satisfies \eqref{ineq:D1conv}, whereas the upper bound on $R_2$ follows from Lemma \ref{prop: bound l}. We get the upper bound
\begin{align*}
D_2 \lesssim \exp \left(a \epsilon U \right) \left( \left( \epsilon U \right)^{\frac{1}{2}} \exp ( \widetilde C U^2 ) \left(N_T^{-\frac{m}{2(m+2)}} + UN_T^{-\frac{m}{2(m+1)}} \right) 
+ \frac{\exp (-C_V (\frac{\epsilon}{2}U)^2 )}{\frac{\epsilon}{2}U}\right),
\end{align*}
which also works for 
$\mathbb{E}[\Delta_{a,1}]^{\frac{1}{2}}$. 
The choice $U^2 = c_u \log N_T$ for 
$c_u$ 
as in \eqref{eq: choice cu} gives us 
\begin{equation}\label{ineq:Delta1}
{\rev \left(\mathbb{E}[\Delta_{a,1}] \right)^{\frac{1}{2}}} \lesssim \exp \left( a \epsilon  (c_u \log N_T)^{\frac{1}{2}} \right) (\log N_T)^{\frac{1}{4}} N_T^{-\gamma} =: \lambda_{N,T}
\end{equation}
for $\gamma$ as in \eqref{eq: gamma}.
Finally, from \eqref{ineq:intd}, \eqref{ineq:Delta} and \eqref{ineq:Delta1} it follows that 
$$ 
{\rev \left(\E \left[ \int_{-\infty}^\infty \left|\delta_{N,T}(y)\right|^2 dy\right] \right)^\frac{1}{2}} \lesssim \frac{\lambda_{N,T}}{\varepsilon_{N,T}} + \varepsilon_{N,T},
$$ 
where 
note that $\varepsilon_{N,T}$ 
can be chosen such that $\varepsilon_{N,T} := \lambda_{N,T}^{\frac{1}{2}} \rightarrow 0$ for $N, T \rightarrow \infty$. It yields 
$$ {\rev \left(\E \left[ \int_{-\infty}^\infty \left|\delta_{N,T}(y)\right|^2 dy\right] \right)^\frac{1}{2}}  \lesssim \lambda^{\frac{1}{2}}_{N,T}$$
as required.

{\rev \subsection{Proof of Theorem \ref{ThLowe}}
We will apply the two hypotheses method described in \cite[Theorem 2.2]{T09}. For this purpose we introduce functions
\begin{align*}
f_0(x) &:= \frac{\imath}{\sqrt{2\pi }} \exp\left(-\frac{x^2}{2}\right)\left(\exp(\imath x)-\exp(-\imath x)\right), \\
f_{\delta,m}(x) &:= \frac{\imath\delta}{\sqrt{2\pi }} \exp\left(-\frac{x^2}{2}\right)\left(\exp(\imath mx)-\exp(-\imath mx)\right).
\end{align*}
The quantities  $\delta\to 0$ and $m\to \infty$ will be chosen later. Note that both functions are real-valued and odd. We set 
\[
W'_0(x):= f_0(x) \qquad \text{and} \qquad 
W'_1(x):= f_0(x) + f_{\delta,m}(x). 
\]
Consequently, the interaction potentials $W_0$ and $W_1$ are even functions. We fix the same confinement potential $V$ in both cases as suggested in Section \ref{secLowe}. Finally, we associate the density function $\pi_k$ (resp. probability measure $\mathbb P_k$) with the pair 
$(V,W_k)$ for $k=0,1$.

First, we check that $W_0,W_1\in \mathcal{A}_{r,p,J}$ for some $(p,r)$ under appropriate conditions on parameters $(\rho,m)$. For this purpose we set 
\begin{align} \label{deandm}
\delta=N^{-1/4} \qquad \text{and} \qquad m=c\delta^{-\frac{1}{J+2}},
\end{align}
for some constant $c>0$. We obtain that
\begin{align*}
f'_0(x) &= \frac{\imath}{\sqrt{2\pi }} \exp\left(-\frac{x^2}{2}\right)\left(x(\exp(-\imath x)-\exp(\imath x))
+ \imath (\exp(\imath x)+\exp(-\imath x)) \right), \\
f'_{\delta,m}(x) &= \frac{\imath\delta}{\sqrt{2\pi }} \exp\left(-\frac{x^2}{2}\right)\left(x(\exp(-\imath mx)-\exp(\imath mx))
+ \imath m(\exp(\imath mx)+\exp(-\imath mx)) \right).
\end{align*}
Since $\delta m\to 0$ due to \eqref{deandm}, we deduce that 
\[
\inf_{x\in \mathbb{R}} W_k''(x) \in [-r_1,-r_2], \qquad k=0,1,
\]
for some $r_1,r_2>0$.
On the other hand, we deduce for some constant $C>0$ that
\begin{align*}
\exp\left(x^2\right) \int_{|y|>x} |W'_0(y)|^2 dy \leq C \exp\left(x^2\right)
\int_{|y|>x} \exp\left(-y^2\right) dy =: r_3.
\end{align*}
As $\delta \to 0$, we conclude that 
\[
\limsup_{x\to \infty}  \exp(2px^2) \int_{|y|>x} |W'_k(y)|^2 dy \leq r_3,
\qquad k=0,1,
\]
for $p=1/2$. The condition 
\[
r_5(1\wedge |z|^{-J-2})\leq |\mathcal{F}(\pi_k)(z)| \leq r_6(1\wedge |z|^{-J-2}), \qquad k=0,1,
\]
is ensured by the choice of the confinement potential $V(x)=\alpha x^2/2+\widetilde V(x)$ (cf. Example  
\ref{exmp}). Finally, the Fourier transforms of $f_0$
and $f_{\delta,m}$ are given as
\begin{align*}
\mathcal F(f_0)(z) &= \imath \left( \exp\left(-\frac{(x+1)^2}{2}\right) -  \exp\left(-\frac{(x-1)^2}{2} \right)\right), \\[1.5 ex]
\mathcal F(f_{\delta,m})(z) &= \imath\delta \left( \exp\left(-\frac{(x+m)^2}{2}\right) -  \exp\left(-\frac{(x-m)^2}{2} \right)\right).
\end{align*}
In view of these identities, we obtain 
\begin{align*}
\int_{\mathbb{R}} \left| \frac{\mathcal{F}(f_0)(z)}{\mathcal{F}(\pi_0)(z)} \right|^2 dz \leq r_5 \int_{\mathbb{R}} \left|\mathcal{F}(f_0)(z) \right|^2
\left((1 \vee |z|^{2(J+2)}\right)dz\leq C_{r_5}
\end{align*}
and, due to \eqref{deandm}, 
\begin{align*}
\int_{\mathbb{R}} \left| \frac{\mathcal{F}(f_{\delta,m})(z)}{\mathcal{F}(\pi_1)(z)} \right|^2 dz &\leq r_5 \int_{\mathbb{R}} \left|\mathcal{F}(f_{\delta,m})(z) \right|^2
\left((1 \vee |z|^{2(J+2)}\right)dz \leq C_{r_5} \delta^2 m^{2(J+2)} \\
&\leq C_{c,r_5} 
\end{align*}
for some constants $C_{r_5}, C_{c,r_5}>0$. This proves that $W_0,W_1 \in \mathcal{A}_{r,p,J}$ for some $(p,r)$. 
Next, we compute the norm $\|W'_0 - W'_1\|_{L^2(\R)}$. We obtain that 
\begin{align*}
\int_{\R} f_{\delta,m}^2(x) dx = \frac{\delta^2}{2\pi} 
\int_{\R} \exp(-x^2) \left(2 - \exp(2\imath mx) - \exp(-2\imath mx) \right) dx
> \underline{c} \delta^2
\end{align*}
for some constant $ \underline{c}>0$. Hence, we deduce that 
\[
\|W'_0 - W'_1\|_{L^2(\R)} > \underline{c}^{1/2} \delta.
\]
In the last step, we compute the Kullback-Leibler divergence 
$K(\mathbb{P}_1^{\otimes N}, \mathbb{P}_0^{\otimes N})$. Following the arguments of the proof of \cite[Theorem 5.1]{BPP}, if $\|W'\|_{\infty}
<\sqrt{C_V-C_W}$, we deduce the inequality
\[
K(\mathbb{P}_1^{\otimes N}, \mathbb{P}_0^{\otimes N}) \lesssim
N \int_{\R} \left(f_{\delta,m} \star \pi_0 \right)^2(x) \pi_1(x) dx.
\]
Consequently, we conclude that 
\begin{align*}
K(\mathbb{P}_1^{\otimes N}, \mathbb{P}_0^{\otimes N}) &\lesssim
N \int_{\R} \left(f_{\delta,m} \star \pi_0 \right)^2(x) dx
= N \int_{\R} |\mathcal F (f_{\delta,m})(z)|^2  |\mathcal F(\pi_0)(z)|^2  dz \\[1.5 ex]
&\lesssim N \int_{\R} |\mathcal F (f_{\delta,m})(z)|^2 
\left(1 \wedge |z|^{-2(J+2)} \right) dz \lesssim N \delta^2 
m^{-2(J+2)}.
\end{align*}
Due to \eqref{deandm}, we finally deduce that 
\[
K(\mathbb{P}_1^{\otimes N}, \mathbb{P}_0^{\otimes N}) \lesssim
N\delta^4 = 1. 
\]
The statement of Theorem \ref{ThLowe} now follows from \cite[Theorem 2.2]{T09}. }

\section{On the Fourier transforms}{\label{s: transform}}
We have observed that \cref{ass: FTs} is crucial in order to obtain the polynomial convergence rate stated in Theorem \ref{th: estim beta}. However, this assumption may appear somewhat unclear at first. The objective of this section is to investigate sufficient conditions that guarantee the validity of \cref{ass: FTs}. 
Given our focus on the super-smooth case, it seems natural to require that the transform of the function $W'$ we aim to estimate decays exponentially fast. 
Verifying the condition \eqref{eq: ass transforms}, however, entails seeking a lower bound for the transform of $\pi$, which will be the main objective of this section. We will thus commence by examining the properties of the Fourier transform of $\pi$, with the clear intention of studying the set of zeros and finding a lower bound outside that set. Remark that it is the same whether considering the transform of the distribution $\Pi$ or the density $\pi$, and it is
$$\mathcal{F}(\Pi) = \mathcal{F}(\pi) = \int_\R \exp(\imath zx)\pi(x) dx.$$

\subsection{Properties of $\mathcal F(\pi)$}
We will use tools from complex analysis to study $\mathcal F(\pi)$ as a function defined on $\mathbb C$. More specifically, our aim is to represent the entire function $\mathcal F(\pi)$ via the Hadamard factorisation theorem. For definitions of the terminology used in this work, we refer to \cite{Ho73}. We begin by deriving an upper bound of the order of $\mathcal F(\pi)$, which stems from the fact that $\pi$ has Gaussian tails. {\rev Let us recall that the order of an entire function $f(z)$ is the infimum of all $m$ such that $f(z) = O(\exp(|z|^m))$ as $z \rightarrow \infty$.}

\begin{theorem}{\label{th: entire}}
Let \cref{ass: beta} hold and $\pi$ be as in \cref{eq: pi}. The map
\begin{equation}\label{eq: FTpi}
    \mathcal F(\pi)\colon\mathbb C\to\mathbb C,\quad z\mapsto \int_\R \exp(\imath zx)\pi(x)\textnormal{d} x
\end{equation}
is an entire function which coincides with the characteristic function of $\pi$ on $\mathbb{R}$. Moreover, the order of $\mathcal F(\pi)$ does not exceed 2.
\end{theorem}

\begin{proof}
The integral in \eqref{eq: FTpi} exists and defines a continuous function at any $z = a + \imath b$ with $a,b \in \mathbb{R}$ since
$|\mathcal{F}(\pi)(a+\imath b)| \le \int_{\mathbb{R}} \exp (- b x) \pi (x) d x  < \infty$
by using $\pi (x) \lesssim \exp (-C_V x^2)$ and $\|W' \star \pi \|_\infty \le \| W' \|_\infty <\infty$. 
Furthermore,
we have $\int_{C}\mathcal F(\pi)(z)\,dz=0$ for any closed contour $C$
which follows from the Cauchy theorem applied to the function $z\mapsto\exp( \imath zx).$
By Morera's theorem, $\mathcal F(\pi)$ is an entire function. To determine the order of $\mathcal F(\pi)$, we use the inequality
\begin{equation*}
    \abs{zx}\le\frac12 \left(\frac{1}{c} \abs z^2+ c \abs x^2 \right) \quad \forall z,x\in\mathbb C,\ c>0,
\end{equation*}
such that, for any $z\in\mathbb C$,
\begin{align*}
    \abs{\mathcal F(\pi) (z)}&\le\int_\R \exp(|\operatorname{Im}(z)x|) \pi(x) d x\\
    &\le \exp\left( \frac{\operatorname{Im}(z)^2}{2c} \right) \int_\R \exp \left( \frac{c x^2}{2} \right) \pi(x) d x.
\end{align*}
Recall that according to \cref{l: bound pi}, $\pi(x) \lesssim \exp( -C_V x^2)$. Hence, the above integral is finite if $x\mapsto \exp(-(C_V-c/2)x^2)$ is integrable on $\R$, which is the case for $C_V>c/2$. As a result, the order of $\mathcal F(\pi)$ does not exceed 2.

\end{proof}

Now that we have that $\mathcal F(\pi)$ is an entire function of finite order, we will find an expression for the function using Hadamard's factorisation theorem. For this, we denote the zeros of $\mathcal F(\pi)$ as $(a_j)_{j\in\mathbb N}$ and order them by increasing modulus. Note that the zeros are symmetric around the imaginary axis, in other words, if $a_j$ is a zero of $\mathcal F(\pi)$, then so is its negative conjugate $-\overline a_j$.
Moreover, $\mathcal F(\pi)$ has no zeros on the imaginary axis because $\mathcal F(\pi) (\imath a) > 0$ for all $a\in\R$, see \cite[Corollary 1 to Theorem 2.3.2]{LiOs77}. 
Let us introduce the critical exponent of convergence $\rho_1$ of the sequence $(a_j)_{j\in\mathbb N}$:
$$
\rho_1= \inf \left\{ r > 0 : \sum_{j\in\mathbb N} 
\frac{1}{|a_j|^{r}} < \infty \right\}.
$$
We denote the order of $\mathcal F(\pi)$ as $\rho$ and make the following
\begin{assumption}\label{ass: rho}
    $\mathcal F(\pi)$ satisfies either $\rho_1 < \rho$ or $\rho_1 = \rho < 2$.
\end{assumption}

\noindent We will use the Hadamard canonical factors
\begin{align*}
    E_d (z) = 
    \begin{cases}
    1-z, &d=0,\\
    (1-z) \exp(z), &d=1,
    \end{cases}
\end{align*}
defined for $z\in\mathbb C$, to study the infinite product representation of $\mathcal F(\pi)$. 

\begin{theorem}
Let \cref{ass: beta} and \cref{ass: rho} hold and $\pi$ be as in \eqref{eq: pi} and $\mathcal F(\pi)$ as in \eqref{eq: FTpi}. Then, there exist $p_1\in\mathbb R$ and $p_2 \ge 0$ such that for all $z\in\mathbb C$,
\begin{equation}\label{eq:factor}
\mathcal F(\pi) (z) = \exp(-p_2 z^2 + \imath p_1 z) \prod_{j\in\mathbb N} E_1 \left( \frac{z}{a_j} \right).
\end{equation}
\end{theorem}
\begin{proof}
Firstly, we consider the case $\rho_1 < \rho$. 
Then $\rho$ is either $2$ or $1$ by \cite[Lemma~4.10.1]{Ho73}. 
The representation in \eqref{eq:factor} follows from \cite[Remark, page 42]{LiOs77}. 
We note that if $\mathcal F(\pi)$ is of order $\rho = 2$, then $p_2 > 0$ because $p_2 = 0$ would lead to the contradiction $\rho \le \max(1, \rho_1)$ by \cite[Theorems 1.2.5, 1.2.7, 1.2.8]{LiOs77}. However, 
if $\mathcal F(\pi)$ is of order $\rho =1$ then its representation
may be reduced to
\begin{equation}\label{def:tildep}
\mathcal F(\pi) (z) = \exp \left( \imath \widetilde p_1 z \right) \prod_{j} 
E_0 \left(\frac{z}{a_j}\right), \qquad \widetilde p_1 := p_1 + \sum_{j} \operatorname{Im} \left( \frac{1}{a_j} \right).
\end{equation}

\noindent Next, we turn to the case $\rho_1 = \rho < 2$.
According to Hadamard's factorization theorem, 
$$
\mathcal F(\pi) (z) = \exp (q_1 z + q_0) \prod_{j} E_1 \left(\frac z{a_j}\right) 
$$
for some $q_1, q_0 \in \mathbb{C}$.
We note that $q_0 = 0$ since $\mathcal F(\pi) (0) = 1$. In order to say something about $q_1$, we note that $X$ with a probability density function $\pi$ satisfies $\E [|X|]< \infty$. Moreover, $\E [X] =0$, whence
\begin{align*}
    0 = \operatorname{Re}\left( \frac{\mathcal F(\pi)'(0)}{\mathcal F(\pi)(0)}\right) = \operatorname{Re} \left( \log \mathcal F(\pi) (z)  \right)' |_{z = 0}
    =  \left. \operatorname{Re}\left(q_1 + \sum_{j} \frac{z}{ a_j (z-a_j)} \right) \right|_{z=0} = \operatorname{Re} (q_1).
\end{align*}
We conclude that \eqref{eq:factor} holds true with $p_2 = 0$, $p_1 = \operatorname{Im}(q_1)$. 
\end{proof}

\begin{remark} \rm
    We have $p_2 >0$ only if $\rho_1 < \rho = 2$. If $\mathcal F(\pi)$ has no zeros then $\pi$ must be a density of a normal distribution with mean zero by \cite[Corollary to Theorem 2.5.1]{LiOs77}, since $\rho$ is finite. \qed
\end{remark}
 The following step consists in bounding the infinite product in \eqref{eq:factor}. 
If $\lfloor \rho_1 \rfloor = 0$, the term $\prod_{j\in\mathbb N}\exp ( z/a_j )$ is well-defined, and the representation becomes
\begin{equation*}
    \mathcal F(\pi)(z)=\exp (-p_2z^2+ \imath \widetilde p_1z )\prod_{j\in\mathbb N} E_0 \left(\frac z{a_j}\right),
\end{equation*}
where $\widetilde p_1$ is the same as in \eqref{def:tildep}.
Let us now use \cite[Lemma 4.12]{Du17}. For all $z\in\mathbb C$ outside $\cup_j B_{\epsilon_j} (a_j)$, where $\epsilon_j = 1/|a_j|^{\rho_1+\varepsilon}$ and $\varepsilon>0$, $\rho_1 + \varepsilon \le 2$, we get
\begin{equation*}
    \abs{\prod_{j\in\mathbb N} E_{\lfloor \rho_1 \rfloor} \left(\frac{z}{a_j}\right)}\gtrsim \exp (-c_\pi |z|^{\rho_1 + \varepsilon} ).
\end{equation*}
It leads us to the following theorem.

\begin{theorem}{\label{th: lower bound}}
Let \cref{ass: beta} and \cref{ass: rho} hold.  Then there exist $c_\pi > 0$ and a family of positive numbers $(\epsilon_j)_{j\in\mathbb N}$ such that, for any $z \in \mathbb{C}$  
outside $\cup_{j \ge 1} B_{\epsilon_j}(a_j)$, it holds that
\begin{equation}{\label{eq: lb pihat}}
|\mathcal F(\pi)(z)| \gtrsim \exp(- c_\pi |z|^2).
\end{equation}  
\end{theorem}
{
\begin{remark} \rm 
Note that without any additional assumption on \(\mathcal F(\pi)\)  one can prove a rougher bound of the form 
\begin{equation}{\label{eq: lb pihat-gen}}
|\mathcal F(\pi)(z)| \gtrsim \exp(- c_\pi |z|^{2+\varepsilon}).
\end{equation} 
for \(z\) outside $\cup_{j \ge 1} B_{\epsilon_j}(a_j)$ where \(\varepsilon\) is an arbitrary positive number. This follows from the fact that the order of $\mathcal F(\pi)$ does not exceed 2 and some estimates for the canonical products, see Section 12.1 in \cite{levin}. As a result, the condition 
\eqref{eq: ass transforms} is satisfied if $|\mathcal F(W')|\lesssim \exp(- c|z|^{2+\varepsilon})$ for some \(c>c_\pi.\)
\qed
\end{remark}
The above theorem can be combined with the Hardy Uncertainty Principle (HUP). HUP is a fundamental result in harmonic
analysis and mathematical physics, extending the ideas of the classical
Heisenberg Uncertainty Principle. Introduced by the British mathematician
G. H. Hardy in 1933, it provides a precise condition under which a
function and its Fourier transform cannot both decay too rapidly unless
the function is identically zero or a specific type of Gaussian function. The
principle can be stated as follows.
\begin{theorem}
Let $f$ be a function in $L^{2}(\mathbb{R})$, and let $\mathcal{F}(f)$ denote its Fourier transform.
Suppose there exist positive constants $\alpha$ and $\beta$ such
that:
\[
|f(x)|\leq C\exp(-\alpha x^{2})\quad\text{and}\quad|\mathcal{F}(f)(u)|\leq C\exp(-\beta u^{2})
\]
for all real numbers $x$ and $u$, where $C$ is a positive constant.
Then:
\begin{enumerate}
\item If $\alpha\beta>\frac{1}{4}$: The only solution is the trivial function
$f(x)=0$ almost everywhere. 
\item If $\alpha\beta=\frac{1}{4}$: The function $f(x)$ must be a Gaussian
function, specifically of the form: 
\[
f(x)=C\exp(-\alpha x^{2})
\]
where $C$ is a constant. 
\item If $\alpha\beta<\frac{1}{4}$: There exist non-trivial functions that
satisfy the decay conditions.
\end{enumerate}
\end{theorem}
This result gives us an upper bound on \(c_\pi \) in \eqref{eq: lb pihat}. Indeed,  if $\cup_{j \ge 1} B_{\epsilon_j}(a_j) \cap \mathcal{L}_{0} = \emptyset$, that is, all zeros of \(\mathcal F(\pi)\) are outside real line, then \(C_{V}c_{\pi}<1/4\) since
\[
\pi(x)\le c_{2}\exp(-C_{V}x^{2})
\]
due to Lemma~\ref{l: bound pi}. This implies the upper bound \(c_{\pi}<1/(4C_V).\) }

 Thanks to Theorem \ref{th: lower bound}, we can explicitly describe a scenario where condition \eqref{eq: ass transforms} is satisfied, assuming that we are in the super-smooth case where the tails of the transform of the interaction function are exponential.

Let us introduce slowly varying functions ${\rev s}(t)$, for $t \ge 0$ (see also \cite{SlowVarying}). These are positive and measurable functions such that, for each $\lambda > 0$, the following holds as $t \rightarrow \infty$: ${\rev s}(\lambda t)/{\rev s}(t) \rightarrow 1$.

\begin{assumption}
    \label{ass: beta exp decay}
    Let $(a_j)_{j\in\mathbb N}$, $c_\pi$, and $(\epsilon_j)_{j\in\mathbb N}$ be as in \cref{th: lower bound}. Assume that $\cup_{j \ge 1} B_{\epsilon_j}(a_j) \cap \mathcal{L}_{0} = \emptyset$. Furthermore, there is a slowly varying function ${\rev s}$ such that $\liminf_{z \rightarrow \infty } {\rev s}(z) > c_\pi$ for which the following holds true on $\mathcal L_{0}$:
    \begin{equation}\label{eq: supersmooth}
        |\mathcal{F}(W')(z)| = O( \exp(- |z|^2 {\rev s} (|z|))) \quad\text{as }\abs{z}\to\infty.
    \end{equation}
\end{assumption}

Note that in the aforementioned assumption, we assert that the transform of $W'$
  exhibits a decay that is almost Gaussian. This observation aligns with our expectations, given the nature of the model under consideration, where the confinement potential is driven by $x^2$. \\
   It is noteworthy to observe that Theorem 5 in \cite{Ass4}, specifically in the scenario where $\bar{a} = 0$, furnishes both necessary and sufficient conditions for the existence of non-trivial functions exhibiting an almost Gaussian nature, accompanied by an almost Gaussian Fourier transform as defined in Equation \eqref{eq: supersmooth}. Furthermore, according to Theorem 1 in \cite{Hardy}, if both $W'$ and its Fourier transform  $\mathcal{F}(W')(z)$ follow the asymptotic behavior $O(|z|^s \exp(-\frac{1}{2}z^2))$ as $z\to \infty$, then each can be expressed as a finite linear combination of Hermite functions. This provides concrete examples that satisfy our assumptions.

\begin{corollary}{\label{cor: transform}}
Let \cref{ass: beta}-\cref{ass: beta exp decay} hold.
Then, we have
$$\int_{\mathcal{L}_{ 0}} \left|\frac{\mathcal{F}\left(W^{\prime}\right)(z)}{\mathcal F(\pi)(z)}\right|^2 d z<\infty.$$
\end{corollary}

\begin{proof}
The corollary is a straightforward consequence of Theorem \ref{th: lower bound} and \cref{ass: beta exp decay}. We have indeed
$$\int_{\mathcal{L}_{0}}\left|\frac{\mathcal{F}\left(W^{\prime}\right)(z)}{\mathcal F(\pi)(z)}\right|^2 d z \le \int_{\mathcal{L}_{0}}\exp (- |z|^2 {\rev s}(|z|) + c_\pi |z|^2 ) dz,$$
which is bounded due to \cref{ass: beta exp decay}. 
\end{proof}

\noindent We can conclude by noting that, under the assumptions of Corollary \ref{cor: transform}, \cref{ass: FTs} is clearly satisfied. This implies that we can achieve a polynomial convergence rate for the estimation of $W'$, as stated in Theorem \ref{th: estim beta}.

{
\begin{example} \label{exmp:nozero}\rm
In some cases we can directly derive  the bound of the form \eqref{eq: lb pihat}. Let $V(x)=\sigma^{2}x^{2}/4$ for some $\sigma^{2}>0.$ Suppose that
$W\in L^{\infty}(\mathbb{R})$ satisfies 
\[
\pi\star W'(x)> 0,\quad \pi\star W''(x)<0,\quad \pi\star W(-x)=\pi\star W(x),\quad x>0.
\]
We have for any $\epsilon\in(0,\sigma),$
\begin{align*}
\mathcal{F}(\pi) & =\frac{c}{ Z_{\pi}}\exp(-\cdot^{2}/(2(\sigma^{2}-\epsilon^{2})))\star\mathcal{F}(f_{\epsilon})
\end{align*}
for some absolute constant $c>0$ where 
\[
f_{\epsilon}(x)=\exp(-x^{2}\epsilon^{2}/2-\pi\star W(x)).
\]
 We have 
\begin{align*}
f_{\epsilon}'(x)&=(-x\epsilon^{2}-\pi\star W'(x))f_{\epsilon}(x)<0,\\[1,5 ex] 
f_{\epsilon}''(x)&=\left[-\epsilon^{2}-\pi\star W''(x)+(x\epsilon^{2}+\pi\star W'(x))^{2}\right]f_{\epsilon}(x)>0,
\end{align*}
for all $x>0$ and small enough $\epsilon>0.$ Hence 
\begin{align*}
\mathcal{F}(f_{\epsilon})(u) & =2\int_{0}^{\infty}\cos(ut)f_{\epsilon}(t)\,dt\\
 & =-\frac{2}{u}\int_{0}^{\infty}\sin(ut)f_{\epsilon}'(t)\,dt\\
 & =\frac{2}{u^{2}}\int_{0}^{\pi}\sin(y)\sum_{k=0}^{\infty}(-1)^{k}\left[-f_{\epsilon}'\left(\frac{y+\pi k}{u}\right)\right]\,dy> 0
\end{align*}
for $u>0.$ As a result \(\mathcal{F}(\pi)(u)>0\) for all \(u\) and
\begin{align*}
\mathcal{F}(\pi)(u) & \geq \frac{c}{Z_{\pi}}\int_{0}^{u}\exp(-t^{2}/(2(\sigma^{2}-\epsilon^{2})))\mathcal{F}(f_{\epsilon})(u-t)\,dt\\
 & \geq \frac{c}{Z_{\pi}} \exp(-u^{2}/(2(\sigma^{2}-\epsilon^{2})))\int_{0}^{u}\mathcal{F}(f_{\epsilon})(u-t)\,dt\\
 & \geq c'\exp(-u^{2}/(2(\sigma^{2}-\epsilon^{2})))f_{\epsilon}(0)
\end{align*}
for $c'>0$ and $ u$ large enough. A similar approach can be used in the case of different potentials $V$.  \qed
\end{example}
}

\begin{example} \label{exmp} \rm
Here we provide an example demonstrating how a non-smooth confinement potential leads to an invariant density with a Fourier transform exhibiting polynomial decay. We have that
\begin{equation*}
\pi (x) = \frac{1}{Z_{\pi}} \exp \left( - 2V(x) - W \star \pi (x) \right), 
\end{equation*}
where $V (x) = (\alpha/2) x^2 + \widetilde V(x)$ and $W (x)$ satisfy \cref{ass: beta}. In addition, assume that $W$, $\widetilde V$ are infinitely differentiable respectively on $\mathbb{R}$, $\mathbb{R} \setminus \{0\}$ and there exist $\widetilde V^{(J+1)} (0^\pm) := \lim_{x \to 0^{\pm}} \widetilde V^{(J+1)} (x) < \infty$ such that
$$
\widetilde V^{(J+1)}(0^+) \neq \widetilde V^{(J+1)}(0^-).
$$ 
Furthermore, $\widetilde V^{(j)}$, $2 \le j \le J+2$, are bounded on $\mathbb{R}\setminus \{0\}$. 
It provides that $\pi$ is infinitely differentiable on $\mathbb{R} \setminus \{0\}$  and there exist $\pi^{(J+1)}(0^+) \neq \pi^{(J+1)}(0^-)$.  
Then, iteratively integrating by parts we obtain
\begin{align*}
\mathcal{F} (\pi)(z) 
= \frac{1}{(-\imath z)^{(J+1)}} (I_-(z)+I_+(z)),
\end{align*}
where
\begin{align*}
   I_-(z) := \int_{-\infty}^0 \exp(\imath zx) \pi^{(J+1)}(x) dx,\qquad 
   I_+ (z) := \int_0^\infty \exp(\imath zx) \pi^{(J+1)} (x) dx.
\end{align*}
If we integrate by parts once more, we obtain
\begin{equation}\label{eq:I+}
I_+ (z) = \frac{1}{\imath z} \left( -\pi^{(J+1)}(0^+) - \int_0^\infty \exp(\imath zx) \pi^{(J+2)}(x) dx \right)
\end{equation}
since $\pi^{(J+1)}(x) \to 0$ as $x \to \infty$ and $\pi^{(J+2)}$ is integrable, which in turn follow using the same arguments as in the proof of Lemma \ref{l: bound pi}. By Riemann's lemma, the last integral in \eqref{eq:I+} tends to zero and so
$$
I_+ (z) :=  -\frac{1}{\imath z}\pi^{(J+1)}(0^+) + o \left( \frac{1}{z} \right)
$$
as $z \to \infty$. Clearly an analogous reasoning applies to $I_- (z)$. It yields that for all large enough $z$,
$$|\mathcal{F}(\pi)(z)| \ge 
\frac{c}{|z|^{J + 2}}.$$
Hence, the Fourier transform of $\pi$ has in this case a polynomial decay, as claimed. Note that the absence of zeros on real line can be studied as in Example~\ref{exmp:nozero}. \qed

\end{example}

\section{Proofs of technical results}{\label{s: proof technical}}
This section is devoted to  proofs of our technical lemmas. We start by providing the proof of Lemma \ref{l: bound pi}. 

\subsection{Proof of Lemma \ref{l: bound pi}}
We can write the invariant density $\pi(x)$ in the following equivalent ways:
\begin{align}\label{eq:pi}
\pi (x) 
&= \frac{1}{Z_\pi} \exp ( -  2V(x) - W \star \pi (x)) \\ 
\label{eq:pi0}
&= \frac{1}{Z_0} \exp (-2V_0(x)- W_0 \star \pi(x)),
\end{align}
where $Z_\pi, 
Z_0$ are the normalizing constants and 
\begin{align*}
V_0 (x) &:= 
\int_0^x V'(u)du = V(x)- V(0),\\
W_0 (x) &:= 
\int_0^x W'(u) d u = W (x)- W (0),
\end{align*}
satisfy $W_0(0) = V_0(0) =0$ 
and $W_0'(x) = W'(x)$, $V'_0(x) = V'(x)$ for all $x \in \mathbb{R}$. Note 
$$
W_0 \star \pi (x) = W \star \pi (x) - W(0)
$$
and
$$
W_0 \star \pi (x) - W_0 \star \pi (0)
= \int_0^x W' \star \pi (u) d u
= W \star \pi (x) - W \star \pi (0).
$$
To obtain the upper and lower bounds on $\pi (x)$ let us use the representation \eqref{eq:pi0}.
According to \cref{ass: beta} we have $\| W' \|_1 < \infty$. We deduce that for all $x$,
$$
| W_0 (x) | \le \int_0^{|x|} |W'(u)| d u \le \| W' \|_1,
$$
hence,
$$
|W_0 \star \pi (x)| \le \| W_0 \|_\infty \| \pi \|_1 \le \| W' \|_1.
$$
We get that for all $x$,
\begin{equation}\label{ineq:upperpi}
\pi (x) \le \frac{1}{Z_0} \exp (-  2V_0(x) + \| W' \|_1) \eqsim \exp (- 2 V_0(x) )
\end{equation}
and
$$
\pi (x) \ge \frac{1}{Z_0} \exp (- 2V_0(x) - \| W' \|_1) \eqsim \exp (- 2V_0(x) ).
$$
As we have assumed that there exists a $C_V>0$ such that $V'' \ge C_V$ we obtain, for all $x \ge 0$, 
\begin{align*}
V_0(x) &= \int_0^x (V'(u)-V'(0)) d u + V'(0) x\\  
&\ge   \frac{C_V}{2} x^2 + V'(0) x
\end{align*}
and 
\begin{align*}
V_0(-x) &= \int_{-x}^0 V'(u) d u\\ & = \int_{-x}^0 (V'(0) - V'(u) )d u  - V'(0) x\\
& \ge  \frac{C_V}{2} x^2 - V'(0)x.  
\end{align*}
In conclusion, for all $x \in \mathbb{R}$, we deduce that 
\begin{equation}\label{ineq:lowerV0}
    V_0 (x) \ge \frac{C_V}{2} x^2 + V'(0)x.
\end{equation}
Recall that $V'(0) = 0$ under both A1 and A2 as $\widetilde{V} = 0$ or $\widetilde{V} \in C^2 (\mathbb{R})$ is assumed to be even. 
The proof about the upper bound of $\pi$ is therefore concluded. \\
Regarding the lower bound, we have that
\begin{equation}{\label{eq: lower start}}
 \pi(x) \ge c \exp (-\alpha x^2-2\widetilde V(x)).   
\end{equation}
Observe that, similarly as above, we can take advantage of the fact that $\tilde{V}'' \le \tilde{c}_2$ to obtain, for all $x \ge 0$,
\begin{align*}
\widetilde{V}(x) -\widetilde V(0)
& = 
\int_0^x (\widetilde{V}'(u) - \widetilde{V}'(0))  du + \widetilde{V}'(0) x \\
& \le 
\tilde{c}_2 \int_0^x u d u  = 
\frac{\tilde{c}_2}{2}x^2,
\end{align*}
having also used that $\widetilde{V}'(0) = 0$. An analogous reasoning holds true for $x \le 0$. It implies that, for any $x \in \R$, $\widetilde{V}(x) \le \widetilde{V}(0) + \frac{\tilde{c}_2}{2} x^2$. Replacing it in \eqref{eq: lower start} we obtain 
$$\pi(x) \ge c \exp(- \alpha x^2 - \tilde{c}_2 x^2 - 2 \widetilde{V}(0) ) = c \exp(- \widetilde{C} x^2)$$
with $\widetilde{C}= \alpha + \tilde{c}_2$, as we wanted. 
 
Let us move to the proof of the upper bound on the derivatives of $\pi$. More specifically, we want to prove by induction that for every natural number $n \le J$ there exists $c>0$ such that 
\begin{equation}{\label{eq: bound deriv pi2 0.5}}
|\pi^{(n)}(x)| \le c (1 + |x|)^{n } \pi (x).
\end{equation}
 Let us begin with the base case $n=1$. Then $\pi = - \varphi'\pi$, where 
 $\varphi' := 2 V' + W' \star \pi$. Note that $|\varphi'(x)| \le c (|x|+1)$, where $|V'(x)| \le c|x|$ follows from $\| V''\|_\infty < \infty$ and $V'(0) = 0$. 
 Now assume the claim \eqref{eq: bound deriv pi2 0.5} holds for all natural numbers up to and including $n< J$. Then the $(n+1)$-th derivative of $\pi$ is
\begin{equation}{\label{eq: deriv pih for ind 1.5}}
\pi^{(n + 1)} = (\pi')^{(n)} = - \sum_{k = 0}^n \binom{n}{k} \varphi^{(k + 1)} \pi^{(n - k)}. 
\end{equation}
The inductive hypothesis ensures that 
$|\pi^{(n-k)}(x)| \le c (1+|x|)^{(n-k)} \pi (x)$.
Moreover, we have that $|\varphi^{(k+1)}(x)| \le c (1+|x|)$ since
$\| W' \star \pi^{(k)} \|_\infty \le \| W' \|_1 \| \pi^{(k)} \|_\infty < \infty$,
whereas 
$|V^{(k+1)}(x)| \le c(1+|x|)$ follows from $\|V^{(j)}\|_\infty \le c$, $2 \le j \le J$.
We conclude by replacing it in \eqref{eq: deriv pih for ind 1.5}, which yields
\begin{align*}
|\pi^{(n + 1)}(x)| &\le c \sum_{k = 0}^n (1 + |x|) (1 + |x|)^{(n - k)} \pi(x) \le c (1+|x|)^{n+1} \pi(x)
\end{align*}
as we wanted, for $n + 1 \le J$.

\subsection{Proof of Proposition \ref{prop: prop chaos L4}}

For convenience, we omit the dependency on $N$ in our notation $X^{i,N}_t$. We will prove the claim by applying a GrÃ¶nwall-type argument to the function
\begin{equation*}
    y(t) := \mathbb{E}[|X^i_t - \overline X^i_t|^{2p}].
\end{equation*}
We define
$$
\Pi^N_t := \frac{1}{N} \sum_{i=1}^N \delta_{X^i_t}, \qquad \overline \Pi^N_t := \frac{1}{N} \sum_{i=1}^N \delta_{\overline X^i_t}, \qquad \overline \Pi_t := \mathcal{L}(\overline X_t).
$$
Since $(X^i_t)_{t \ge 0}$ and $(\overline X^i_t)_{t \ge 0}$ start at $X^i_0 = \overline X^i_0$ and are driven by the same Brownian motion $(B^i_t)_{t \ge 0}$, it holds that
\begin{align*}
X^i_t - \overline X^i_t = &- \int_0^t (V'(X^i_s) - V'(\overline X^i_s)\\ 
&+ \frac{1}{2} W' \star \Pi^N_s (X^i_s) - \frac{1}{2} W' \star \overline \Pi_s (\overline X^i_s)) d s.
\end{align*}
Applying It\^o's formula, summing over $i=1, \dots, N$ and dividing by $N$, yields 
\begin{align*}
\frac{1}{N} \sum_{i=1}^N |X^i_t - \overline X^i_t|^{2p} = 
&- \frac{2p}{N} \sum_{i=1}^N \int_0^t \left(A_i(s) + \frac{1}{2} B_i(s) + \frac{1}{2} C_i(s)\right) ds,
\end{align*}
where 
\begin{align*}
A_i(s) &:= (X^i_s - \overline X^i_s)^{2p-1} (V'(X^i_s)-V'(\overline X^i_s)),\\
B_i(s) &:= (X^i_s - \overline X^i_s)^{2p-1} (W' \star \Pi^N_s (X^i_s) - W' \star \overline \Pi^N_s (\overline X^i_s)),\\ 
C_i(s) &:= (X^i_s - \overline X^i_s)^{2p-1} (W' \star \overline \Pi^N_s (\overline X^i_s) - W' \star \overline \Pi_s (\overline X^i_s)).
\end{align*}
Taking the expectation and derivative gives 
\begin{equation}{\label{eq: start yprime 1}}
y^\prime(t) = - \frac{p}{N} \sum_{i=1}^N \mathbb{E}[2 A_i(t)+B_i(t)+C_i(t)].
\end{equation}
Using the assumption $V'' \ge C_V >0$ and the mean value theorem gives
\begin{align}{\label{eq: bound A 2}}
- \mathbb{E}[A_i(t)] \le - C_V \mathbb{E}[|X^i_t - \overline X^i_t|^{2p}].  
\end{align}
The analysis of $B_i(t)$ makes use of the symmetry of $W$ and the exchangeability of $(X^i_t, \overline X^i_t)$, $i=1,\dots,N$. Indeed, we obtain 
$$
- \mathbb{E} [B_i (t)]
= - \frac{1}{N} \sum_{j=1}^N \mathbb{E}[B_{ij}(t)],
$$
where 
\begin{align*}
\mathbb{E}[B_{ij}(t)] &= \mathbb{E}[(X^i_t-\overline X^i_t)^{2p-1} (W'(X^i_t-X^j_t) - W'(\overline X^i_t-\overline X^j_t))]\\
&= \frac{1}{2} \mathbb{E}[((X^i_t-\overline X^i_t)^{2p-1}-(X^j_t-\overline X^j_t)^{2p-1})(W'(X^i_t-X^j_t) - W'(\overline X^i_t-\overline X^j_t))].
\end{align*}
By the mean value theorem, the assumption $- W'' \le C_W$ gives
\begin{align*}
-\mathbb{E}[B_{ij}(t)] &\le \frac{C_W}{2} \mathbb{E}[((X^i_t-\overline X^i_t)^{2p-1}-(X^j_t-\overline X^j_t)^{2p-1}) ((X^i_t-X^j_t)-(\overline X^i_t-\overline X^j_t))]\\
&\le 2C_W \mathbb{E}[|X^i_t - \overline X^i_t|^{2p}],
\end{align*}
hence,
\begin{equation}{\label{eq: bound B 3}}
- \mathbb{E} [B_i (t)] \le 2 C_W \mathbb{E}[|X^i_t - \overline X^i_t|^{2p}].  
\end{equation}


\noindent H\"older's inequality for $C_i(t)$ implies 
$$
-\mathbb{E} [C_i(t)] \le \mathbb{E} [|X^i_t - \overline X^i_t|^{2p}]^{\frac{{2p-1}}{2p}} R_{i}(t)^{\frac{1}{2p}},
$$
where 
\begin{align*}
R_i(t) :=& \mathbb{E} [|W' \star \overline \Pi^N_t (\overline X^i_t)- W' \star \overline \Pi_t (\overline X^i_t)|^{2p}]\\
=&\Eb{\abs{\frac{1}{N} \sum_{j=1}^N W' (\overline X^i_t - \overline X^j_t)- W' \star \overline \Pi_t (\overline X^i_t)}^{2p}}.
\end{align*}
Expanding this term, we deduce that 
\begin{align*}
R_i(t) = \frac{1}{N^{2p}} \sum_{j_1, \dots, j_{2p} = 1}^N \Eb{\prod_{k=1}^{2p} \mathbb{E} [ W' (\overline X^i_t - \overline X^{j_k}_t) - W' \star \overline \Pi_t (\overline X^i_t) | \overline X^i_t ]}, 
\end{align*}
because of the independence of $\overline X^i_t$, $i=1,\dots,N$. The key remark to study this expectation is that for $i \neq j$,
$$
\mathbb{E}[W'(\overline X^i_t - \overline X^j_t) - W' \star \overline \Pi_t (\overline X^i_t) | \overline X^i_t] = 0
$$
since $\mathcal{L}(\overline X^j_t) = \overline \Pi_t$. We observe that if there exists $k$ such that $j_k \neq j_{\widetilde k}$ $\forall \widetilde k \in \{1, ... , 2p \} \setminus k$, then the $2p$-fold product vanishes. In other words, in order for a term to contribute to the sum in $R_i(t)$, for every index $j_k$ there must be another $j_{\widetilde k}$ such that $j_k=j_{\widetilde k}$. We can have at most $N^p$ of these combinations of indices.
We recall we have assumed $\| W' \|_\infty < \infty$, which implies $\mathbb{E}[|W'(\overline X^i_t - \overline X^j_t) - W' \star \overline \Pi_t (\overline X^i_t)|^k]<\infty$ for any $k\in\mathbb N$. Therefore, we have
\begin{equation*}
    R_i(t)^{\frac 1{2p}}\lesssim N^{-\frac 12}.
\end{equation*}
In conclusion,
{\rev 
\begin{equation}{\label{eq: bound C 4}}
- \mathbb{E} [C_i(t)] \lesssim \frac{1}{\sqrt{N}} \left( \mathbb{E}[|X^i_t - \overline X^i_t|^{2p}] \right)^\frac{2p-1}{2p}. 
\end{equation}
By exchangeability and replacing \eqref{eq: bound A 2}, \eqref{eq: bound B 3} and \eqref{eq: bound C 4} in \eqref{eq: start yprime 1} we obtain
\begin{align*}
y'(t) \le - 2(C_V-C_W) y(t) + \frac{c}{\sqrt{N}}y(t)^{\frac{2p-1}{2p}}.    
\end{align*}
This is equivalent to 
$$\frac{y'(t)}{y(t)^{\frac{2p-1}{2p}}} \le - 2(C_V-C_W) y(t)^\frac{1}{2p} + \frac{c}{\sqrt{N}}.$$
Let us now define $\beta(t) := y(t)^\frac{1}{2p}$. Observe that $\beta'(t) = \frac{1}{2p} \frac{y'(t)}{y(t)^{\frac{2p-1}{2p}}} $. Then, $\beta(t)$ is solution to 
$$\beta'(t) + \frac{1}{p}(C_V-C_W) \beta(t) \le \frac{c}{\sqrt{N}}.$$
Since $\beta(0) = y(0)=0$ and $C_V-C_W > 0$, the conclusion follows by integrating this Gr\"onwall-like differential inequality, which provides $\beta(t) \le c/\sqrt{N}$ uniformly in $t$.}

\subsection{Proof of Lemma \ref{l: moments}} \label{Sec6.3}
Observe that for $c>0$ and any random variable $X$, 
\begin{align}{\label{eq: exp to pol}}
0< \frac{1}{2}\left(\mathbb{E}[\exp(cX)] + \mathbb{E}[\exp(-cX)]\right) 
&= 1+ \sum_{k=1}^\infty \frac{c^{2k}}{(2k)!} \mathbb{E} [X^{2k}],    
\end{align}
hence, it suffices to study the asymptotic growth rate of its even moments in order to show $\mathbb{E}[\exp(\pm c X)] < \infty$. 

Let us start by looking at $\mathbb{E}[\exp(\pm c X_t^{i,N})]$, which leads us to study $\mathbb{E}[(X_t^{i,N})^{2k}]$ for $k \in \mathbb{N}$. For convenience, we omit the dependency on $N$ in our notation and recall that the particles follow the system of SDE's:
$$
X^i_t = X^i_0 + \int_0^t - V'(X^i_s) - \frac{1}{2N} \sum_{j=1}^N W' (X^i_s-X^j_s) d s + B^i_t, \quad i =1,\dots,N.
$$ 
Applying the It\^o lemma gives
\begin{align*}
(X^i_t)^{2k} = (X^i_0)^{2k} + \int_0^t &- 2k (X^i_s)^{2k-1} (V'(X^i_s) + \frac{1}{2N} \sum_{j=1}^N W'(X^i_s - X^j_s))\\ 
&+ \frac{1}{2} (2k)(2k-1) (X^i_s)^{2k-2} ds + 2k \int_0^t (X^i_s)^{2k-1} d B^i_s.
\end{align*}
Taking the expectation, and then differentiating,
\begin{align*}
\frac{d}{dt} \mathbb{E}[(X^i_t)^{2k}]  = &- 2k \mathbb{E}[(X^i_t)^{2k-1}V'(X^i_t)] - \frac{k}{N} \sum_{j=1}^N \mathbb{E} [ (X^i_t)^{2k-1} W'(X^i_t - X^j_t)]\\ 
&+  k(2k-1) \mathbb{E}[(X^i_t)^{2k-2}],
\end{align*}
where \begin{align}{\label{eq: conv V}}
- \mathbb{E}[(X^i_t)^{2k-1} V'(X^i_t)] &= - \mathbb{E}[(X^i_t)^{2k-1}(V'(X^i_t)-V'(0))] - \mathbb{E}[(X^i_t)^{2k-1} V'(0)]\nonumber\\ 
&\le - C_V \mathbb{E} [|X^i_t|^{2k}] + |V'(0)|\mathbb{E}[|X^i_t|^{2k-1}],
\end{align}
follows from the strong convexity of $V$, i.e.\ $V''\ge C_V > 0$, whereas $\| W' \|_\infty < \infty$ implies
$$
- \mathbb{E}[(X^i_t)^{2k-1} W' (X^i_t - X^j_t)] \le \| W' \|_\infty \mathbb{E}[|X^i_t|^{2k-1}].
$$
We get
\begin{align}
\frac{d}{dt} \mathbb{E} [|X^i_t|^{2k}] \le &- 2 k C_V \mathbb{E} [|X^i_t|^{2k}] + 2k |V'(0)|\mathbb{E}[|X^i_t|^{2k-1}] + k\| W' \|_\infty \mathbb{E}[|X^i_t|^{2k-1}]\nonumber\\ \label{ineq:m2k}
&+ k (2k-1) \mathbb{E} [|X^i_t|^{2k-2}],
\end{align}
where 
\begin{align*}
\mathbb{E}[|X^i_t|^{2k-l}] &= \mathbb{E}[|X^i_t|^{2k-l} (\one(|X^i_t|\le C) + \one(|X^i_t|>C))] \\
&\le C^{2k-l} + C^{-l} \mathbb{E}[|X^i_t|^{2k}], \quad l = 1,2, 
\end{align*}
for all $C >0$. 
Denote $m_t(2k) := \mathbb{E}[|X^i_t|^{2k}]$. 
Then it holds that
\begin{align*}
\frac{d}{dt}m_t(2k) &\le -2kC_V m_t(2k) + k (2|V'(0)| + \| W'\|_\infty) (C^{2k-1} + C^{-1} m_t(2k))\\
&\quad + k (2k-1) (C^{2k-2} + C^{-2} m_t(2k))\\
&= - A k (m_t (2k) - B),
\end{align*}
where
\begin{align*}
A &:= 2C_V - (2|V'(0)| + \| W' \|_\infty) C^{-1} - (2k-1) C^{-2},\\ 
B &:= A^{-1} ((2|V'(0)| + \| W' \|_\infty) C^{2k-1} +  (2k-1) C^{2k-2})
\end{align*}
do not depend on $t$. For some fixed $\varepsilon \in (0,1)$, set
$$
C = C(k) := \left( \frac{(2k-1)}{2C_V \varepsilon} \right)^{\frac 1 2}, \quad k \in \mathbb{N},
$$ 
so that 
\begin{align*}
A = A(k) \sim 2 C_V (1-\varepsilon) >0, \qquad B = B(k) \sim \frac{2C_V \varepsilon C(k)^{2k}}{A(k)},
\end{align*}
where $\sim$ denotes asymptotic equality as $k \to \infty$. The Gr\"onwall inequality 
$$
\frac{d}{dt} m_t(2k) = \frac{d}{dt} (m_t(2k)-B(k)) \le - A(k) k (m_t(2k)-B(k))
$$
gives
\begin{align*}
    m_t(2k) 
    \le B(k) + (m_0(2k) - B(k)) e^{- A(k) k t},
\end{align*}
which in turn implies that, for all large enough $k$, uniformly in $t$,
$$
m_t(2k) \le \max (m_0(2k),B(k)).
$$
Based on \eqref{eq: exp to pol} and the subsequent analysis, we conclude that $\sup_{t \ge 0} \mathbb{E}[ \exp (\pm c X_t^i)]$ is bounded as long as
$$\sup_{t \ge 0} \sum_{k = 1}^\infty \frac{c^{2k}}{(2k)!} m_t(2k) \lesssim \sum_{k = 1}^\infty \frac{c^{2k}}{(2k)!} (m_0(2k)+B(k))$$
is also bounded.
We observe that $\sum_{k = 1}^\infty c^{2k} m_0(2k)/(2k)!$ converges because of \eqref{eq: exp to pol} and 
our assumption
$\mathbb{E}[\exp (\pm c X_0^i)] < \infty$.
To test the convergence of the series $\sum_{k=1}^\infty c^{2k} B(k)/(2k)!$
 we use the ratio test also known as d'Alembert's criterion: 
\begin{align*}
\frac{c^{2k+2} B(k+1)/(2k+2)!}{c^{2k}B(k)/(2k)!}
&\sim 
\frac{c^2 C(k+1)^{2k+2}}{C(k)^{2k} (2k+1)(2k+2)}\\
&= \frac{c^2 (2k+1)^{k+1}}{(2C_V \varepsilon) (2k-1)^k (2k+1) (2k+2)} \frac{}{}  
 \to 0
\end{align*}
as $k\to \infty$. The result for the exponential moments of $X_t^i$ follows. \\
\\
One can follow the preceding proof to get $\sup_{t \ge0} \mathbb{E}[\exp(\pm c \overline X_t )]<\infty$ from $\mathbb{E}[\exp (\pm c \overline X_0)] < \infty$, where recall $(\overline X_t)_{t \ge 0}$ is a solution of
$$
\overline X_t = \overline X_0 - \int_0^t V'(\overline X_s) + \frac{1}{2}  W'\star \overline \Pi_s (\overline X_s) d s +  B_t
$$
with $\overline \Pi_t := \mathcal{L}(\overline X_t)$. Indeed, It\^o's lemma for $(\overline X_t)^{2k}$ gives 
\begin{align*}
\frac{d}{dt} \mathbb{E}[(\overline X_t)^{2k}] = &- 2k \mathbb{E}[(\overline X_t)^{2k-1}(V'(\overline X_t)+\frac{1}{2}W' \star \overline \Pi_t (\overline X_t)]\\ 
&+ k(2k-1) \mathbb{E}[(\overline X_t)^{2(k-1)}],
\end{align*}
where, in the same manner as in \eqref{eq: conv V}, we have
$$
- \mathbb{E}[(\overline X_t)^{2k-1} V'(\overline X_t)] 
\le - C_V \mathbb{E}[|\overline X_i|^{2k}] + |V'(0)| \mathbb{E}[|\overline X^i_t|^{2k-1}],
$$
and since $\|W'\|_\infty<\infty$, we get
\begin{align*}
- \mathbb{E}[(\overline X_t)^{2k-1} W' \star \overline \Pi_t (\overline X_t)] \le \| W' \|_\infty \mathbb{E}[|\overline X_t|^{2k-1}].
\end{align*} 
Thus, we obtain the inequality \eqref{ineq:m2k} for $\mathbb{E}[|\overline X_t|^{2k}]$ instead of $\mathbb{E}[|X^i_t|^{2k}]$. The rest of the proof remains the same.

\backmatter





\bmhead{Acknowledgements}
Denis Belomestny acknowledges financial support from Deutsche Forschungsgemeinschaft (DFG), grant 497300407.
Mark Podolskij and Shi-Yuan Zhou gratefully acknowledge financial support of ERC Consolidator Grant 815703 “STAMFORD: Statistical Methods for High Dimensional Diffusions”.


\begin{thebibliography}{100}

\bibitem{AbrNic} Abraham, K., \& Nickl, R. (2020). On statistical Calder\'on problems. Mathematical Statistics and Learning, 2(2), 165-216.

\bibitem{Us22} Amorino, C., Heidari, A., Pilipauskait\.e, V., \& Podolskij, M. (2023). Parameter estimation of discretely observed interacting particle systems. Stochastic Processes and their Applications, 163, 350-386.


\bibitem{Bal12} Baladron, J., Fasoli, D., Faugeras, O., \& Touboul, J. (2012). Mean-field description and propagation of chaos in networks of Hodgkin-Huxley and FitzHugh-Nagumo neurons. The Journal of Mathematical Neuroscience, 2, 1-50.

\bibitem{BelGol} Belomestny, D., \& Goldenshluger, A. (2021). Density deconvolution under general assumptions on the distribution of measurament errors. The Annals of Statistics, 49(2), 615-649.

\bibitem{BPP} Belomestny, D., Pilipauskait\.e, V., \& Podolskij, M. (2023). Semiparametric estimation of McKean-Vlasov SDEs. Annales de l'Institut Henri Poincar\'e, Probabilit\'es et Statistiques, 59(1), 79-96. 

\bibitem{existence} Benachour, S., Roynette, B., Talay, D., \& Vallois, P. (1998). Nonlinear self-stabilizing processes, I Existence, invariant probability, propagation of chaos. Stochastic processes and their applications, 75(2), 173-201.


\bibitem{ridgedeconv} Bertero, M.,  Boccacci, P. (1998). Introduction to Inverse Problems in Imaging. Taylor \& Francis.

\bibitem{bruijn} Bruijn, de, N. G. (1950). The roots of trigonometric integrals. Duke Mathematical Journal, 17(3), 197-226.


\bibitem{Sni91} Burkholder, D. L., Pardoux, E., \& Sznitman, A. S. (2006). Ecole d\'et\'e de probabilit\'es de Saint-Flour XIX-1989. Springer.
\bibitem{Can12} Canuto, C., Fagnani, F., \& Tilli, P. (2012). An Eulerian approach to the analysis of Krause's consensus models. SIAM Journal on Control and Optimization, 50(1), 243-265.

\bibitem{CMV03} Carrillo, J.A., McCann, R.J, \& Villani, C. (2003). Kinetic equilibration rates for granular media and related equations: entropy dissipation and mass transportation estimates. Rev. Mat. Iberoamericana 19(3), 971-1018.

\bibitem{9McK} Cardaliaguet, P., Delarue, F., Lasry, J. M., \& Lions, P. L. (2019). The Master Equation and the Convergence Problem in Mean Field Games:(AMS-201) (Vol. 201). Princeton University Press.

\bibitem{CatGuiMal08} Cattiaux, P., Guillin, A., \& Malrieu, F. (2008). Probabilistic approach for granular media equations in the non-uniformly convex case. Probability Theory and Related Fields, 140, 19-40.

\bibitem{Cha17} Chazelle, B., Jiu, Q., Li, Q., \& Wang, C. (2017). Well-posedness of the limiting equation of a noisy consensus model in opinion dynamics. Journal of Differential Equations, 263(1), 365-397.

\bibitem{Che21} Chen, X. (2021). Maximum likelihood estimation of potential energy in interacting particle systems from single-trajectory data. Electronic Communications in Probability, 26, 1-13.

\bibitem{ComGen23} Comte, F., \& Genon-Catalot, V. (2023). Nonparametric adaptive estimation for interacting particle systems. To appear in Scandinavian Journal of Statistics.

\bibitem{deBruijn} de Bruijn, N. G. (1950). The roots of trigonometric integrals. Duke Mathematical Journal, 17(3), 197-226.

\bibitem{MaeHof} Della Maestra, L., \& Hoffmann, M. (2022). Nonparametric estimation for interacting particle systems: McKean-Vlasov models. Probability Theory and Related Fields, 182, 551-613.

\bibitem{DelHof22} Della Maestra, L., \& Hoffmann, M. (2023). The LAN property for McKean-Vlasov models in a mean-field regime. Stochastic Processes and their Applications, 155, 109-146.

\bibitem{21McK} Djehiche, B., Gozzi, F., Zanco, G., \& Zanella, M. (2022). Optimal portfolio choice with path dependent benchmarked labor income: a mean field model. Stochastic Processes and their Applications, 145, 48-85.

\bibitem{Doe} Doetsch, G. (2012). Introduction to the Theory and Application of the Laplace Transformation. Springer Science \& Business Media.

\bibitem{Du17} Dupuy, T. (2017) Hadamard theorem and entire functions of finite order -- for Math 331. Lecture notes.
\url{https://tdupu.github.io/complexspring2017/hadamard.pdf}

\bibitem{Fou13} Fouque, J. P., \& Sun, L. H. (2013). Systemic risk illustrated. Handbook on Systemic Risk, 444, 452.

\bibitem{GenCat21a} Genon-Catalot, V., \& Laredo, C. (2021). Probabilistic properties and parametric inference of small variance nonlinear self-stabilizing stochastic differential equations. Stochastic Processes and their Applications, 142, 513-548.

\bibitem{GenCat21b} Genon-Catalot, V., \& Laredo, C. (2021). Parametric inference for small variance and long time horizon McKean-Vlasov diffusion models. Electronic Journal of Statistics, 15(2), 5811-5854.

\bibitem{GenCat23} Genon-Catalot, V., \& Laredo, C. (2023). Inference for ergodic McKean-Vlasov stochastic differential equations with polynomial interactions. To appear in  Annales de l'Institut Henri Poincar\'e, Probabilit\'es et Statistiques. Preprint hal-03866218.

\bibitem{GenCat23b} Genon-Catalot, V., \& Laredo, C. (2023). Parametric inference for ergodic McKean-Vlasov stochastic differential equations. Preprint hal-04071936.

\bibitem{33McK} Giesecke, K., Schwenkler, G., \& Sirignano, J. (2019). Inference for large financial systems. Mathematical Finance, 30(1), 3-46.

\bibitem{Guy11} Guyon, J., \& Henry-Labordere, P. (2011). The smile calibration problem solved. Available at SSRN 1885032.

\bibitem{Hardy} Hardy, G. H. (1933). A theorem concerning Fourier transforms. Journal of the London Mathematical Society, 1(3), 227-231.

\bibitem{HofOli} Hoffmann, M., \& Olivier, A. (2016). Nonparametric estimation of the division rate of an age dependent branching process. Stochastic Processes and their Applications, 126(5), 1433-1471.

\bibitem{Ho73} Holland, A. S. (1974). Introduction to the theory of entire functions. Academic Press.

\bibitem{HuoZha} Huo, X., Zhan, Y. (2021). A note on the entire functions: theorems, properties and examples. In Journal of Physics: Conference Series (Vol. 2012, No. 1, p. 012058). IOP Publishing.

\bibitem{JouMel} Jourdain, B., \& M\'el\'eard, S. (1998, January). Propagation of chaos and fluctuations for a moderate model with smooth initial data. Annales de l'Institut Henri Poincar\'e (B) Probability and Statistics, 34(6), 727-766. 

\bibitem{Kam} Kamynin, I. P. (1982). Generalization of the theorem of Marcinkiewicz on entire characteristic functions of probability distributions. Journal of Soviet Mathematics, 20(3), 2175-2180.

\bibitem{Kas90} Kasonga, R. A. (1990). Maximum likelihood theory for large interacting systems. SIAM Journal on Applied Mathematics, 50(3), 865-875.

\bibitem{levin} Levin, B. Y. (1997). Lectures on entire functions.  Providence, RI: American Mathematical Society.

\bibitem{LiOs77} Linnik, Y. V., Ostrovskii, I. V. (1977). {\it Decomposition of random variables and vectors.}
AMS, Volume 48.

\bibitem{Mai} Ma\"ida, M., Nguyen, T. D., Pham Ngoc, T. M., Rivoirard, V., \& Tran, V. C. (2022). Statistical deconvolution of the free Fokker-Planck equation at fixed time. Bernoulli, 28(2), 771-802.

\bibitem{Mal01} Malrieu, F. (2001). Logarithmic Sobolev inequalities for some nonlinear PDE's. Stochastic Processes and their Applications, 95(1), 109-132.

\bibitem{Mal03} Malrieu, F. (2003). Convergence to equilibrium for granular media equations and their Euler schemes. The Annals of Applied Probability, 13(2), 540-560.

\bibitem{McK66} McKean Jr, H. P. (1966). A class of Markov processes associated with nonlinear parabolic equations. Proceedings of the National Academy of Sciences, 56(6), 1907-1911.

\bibitem{Mel96} M\'el\'eard, S. (1996). Asymptotic behaviour of some interacting particle systems; McKean-Vlasov and Boltzmann models. In Probabilistic Models for Nonlinear Partial Differential Equations, Lecture Notes in Mathematics,
1627, 42-95, Springer.

\bibitem{Mog99} Mogilner, A., \& Edelstein-Keshet, L. (1999). A non-local model for a swarm. Journal of Mathematical Biology, 38, 534-570.

\bibitem{Mon} Monard, F., Nickl, R., \& Paternain, G. P. (2021). Consistent Inversion of Noisy Non-Abelian 
X-Ray Transforms. Communications on Pure and Applied Mathematics, 74(5), 1045-1099.

\bibitem{Nic} Nickl, R. (2020). Bernstein von Mises theorems for statistical inverse problems I: Schr\"odinger equation. Journal of the European Mathematical Society, 22(8), 2697-2750.

\bibitem{PavZan22a} Pavliotis, G. A., \& Zanoni, A. (2022). Eigenfunction martingale estimators for interacting particle systems and their mean field limit. SIAM Journal on Applied Dynamical Systems, 21(4), 2338-2370.

\bibitem{PavZan22b} Pavliotis, G. A., \& Zanoni, A. (2022). A method of moments estimator for interacting particle systems and their mean field limit. arXiv preprint arXiv:2212.00403.

\bibitem{Ass4} Sedletskii, A. M. (2008). Classes of entire functions that are rapidly decreasing on the real axis: theory and applications. Sbornik: Mathematics, 199(1), 131.

\bibitem{SlowVarying} Seneta, E. (2006). Regularly varying functions (Vol. 508). Springer.


\bibitem{Imp21} Sharrock, L., Kantas, N., Parpas, P., \& Pavliotis, G. A. (2023)
Online parameter estimation for the McKean-Vlasov stochastic differential equation. Stochastic Processes and their Applications, 162, 481-546.

\bibitem{T09} A. B. Tsybakov (2009). {\em Introduction to nonparametric estimation}. Springer Science+Business Media, New York.

\bibitem{Wid46} Widder, D. V. (1946). The Laplace transformation. Princeton University Press.

\bibitem{JJ09} Johannes, J. (2009). Deconvolution with unknown error distribution. Ann. Statist. 37 (5A) 2301 - 2323.
\end{thebibliography}
\end{document}